\documentclass{amsart}

\usepackage{amsmath}
\usepackage{graphicx}
\usepackage{textcomp}
\usepackage{amsbsy}
\usepackage{latexsym}
\usepackage[mathscr]{eucal}
\usepackage{amsfonts,amsmath,amsthm}
\usepackage{amssymb}
\usepackage{diagbox}
\usepackage{hhline}
\usepackage{xcolor}
\usepackage{tikz}
\usepackage{gnuplot-lua-tikz}
\usepackage{stmaryrd}

\usepackage{comment}
\usepackage{fullpage}

\def\rnew{\color{magenta}}
\def\cnew{\color{blue} }
\newtheorem{lemma}{Lemma}[section]

\newtheorem{theorem}[lemma]{Theorem}
\newtheorem{remark}[lemma]{Remark}

\def\RR{\rm \hbox{I\kern-.2em\hbox{R}}}
\def\NN{\rm \hbox{I\kern-.2em\hbox{N}}}
\def\ZZ{\rm {{\rm Z}\kern-.28em{\rm Z}}}
\def\CC{\rm \hbox{C\kern -.5em {\raise .32ex \hbox{$\scriptscriptstyle
|$}}\kern
-.22em{\raise .6ex \hbox{$\scriptscriptstyle |$}}\kern .4em}}

\def\<{\langle}
\def\>{\rangle}
\def\t{\tilde}
\def\e{\varepsilon}

\def\o{\overline}

\def\cT{{\mathcal T}}

\def\cO{{\mathcal O}}
\def\Chi{\raise .3ex
\hbox{\large $\chi$}} 
\def\lsima{\hbox{\kern -.6em\raisebox{-1ex}{$~\stackrel{\textstyle<}{\sim}~$}}\kern -.4em}
\def\lsim{\hbox{\kern -.2em\raisebox{-1ex}{$~\stackrel{\textstyle<}{\sim}~$}}\kern -.2em}
\def\[{\Bigl [}
\def\]{\Bigr ]}
\def\({\Bigl (}
\def\){\Bigr )}
\def\[{\Bigl [}
\def\]{\Bigr ]}
\def\({\Bigl (}
\def\){\Bigr )}

\def\R{\mathbb{R}}

\def\T{{\relax\ifmmode I\!\!\hspace{-1pt}T\else$I\!\!\hspace{-1pt}T$\fi}}

\def\lsim{\raisebox{-1ex}{$~\stackrel{\textstyle<}{\sim}~$}}

  \def\NN{N}                  

\def\t#1{\tilde{#1}}

\def\cO{{\mathcal O}}

\def\cN{{\mathcal N}}
\def\cT{{\mathcal T}}

\def\cL{{\mathcal L}}

\def\cK{{\mathcal K}}

\newcommand{\be}{\begin{equation}}
\newcommand{\ee}{\end{equation}}
\newcommand{\bea}{$$ \begin{array}{lll}}
\newcommand{\eea}{\end{array} $$}

\newcommand{\beqn}{\begin{equation}}
\newcommand{\eeqn}{\end{equation}}

\def\endproof{\hfill\rule{1.5mm}{1.5mm}\\[2mm]}

\makeatletter
\@addtoreset{equation}{section}
\makeatother

\newcommand\eref[1]{{\rm (\ref{#1})}}
\newcommand{\iref}[1]{{\rm (\ref{#1})}}

\def\int{\intop\limits}

\newcommand\cH{{\mathcal H}}

 \newcommand{\Tr}{\mathbb{T}_h}

\begin{document}

\title{Solving PDEs with Incomplete Information}
\date{\today}
\author{Peter Binev, Andrea Bonito, Albert Cohen, Wolfgang Dahmen \\ Ronald DeVore, and Guergana Petrova}
\thanks{This research was supported by the NSF Grants
DMS 2110811 (AB), DMS 2038080 (PB and WD),  DMS-2012469 (WD), DMS 21340077 (RD and GP), the MURI ONR Grant N00014-20-1-278 (RD and GP), the ARO Grant W911NF2010318 (PB), 
 and the  SFB 1481,
funded by the German Research Foundation (WD)}

\maketitle

  \begin{abstract} 
  We consider the problem of numerically approximating the solutions to a partial differential equation (PDE) when there is 
  insufficient information to determine a unique solution. Our main example is 
  the Poisson boundary value problem,  when the boundary data
  is unknown and instead one observes finitely many linear measurements
  of the solution.  We view this setting as an optimal recovery problem and develop theory
  and numerical algorithms for its solution. The main vehicle employed is the derivation and approximation
  of the Riesz representers of these functionals with respect to relevant Hilbert spaces of harmonic functions.
 \end{abstract}

\section{Introduction}
\label{S:intro}
  The questions we investigate  sit in the 
  broad research area of using measurements     
  to enhance the  numerical recovery of  the solution $u$ to a PDE. 
  The particular setting addressed in this paper is to numerically approximate the solution to an elliptic boundary value problem when there is insufficient information on the boundary value to determine a unique solution to the PDE. In place of complete boundary information, we have a finite number of  data observations of the solution $u$.  This data serves to narrow the set of possible solutions.  We ask what is the optimal accuracy to which we can recover $u$ and what is a near optimal numerical algorithm to approximate $u$. Problems of this particular type arise in several fields of  science and engineering  (see  e.g. \cite{xu2021explore,brunton2020machine,duraisamy2019turbulence} for examples in fluid dynamics), where a lack of full information on boundary conditions arises for various reasons.   For example, the correct physics might not be fully understood \cite{raj,richards2019appropriate}, or  the boundary values are not accessible \cite{grinberg2008outflow}, or they must be appropriately modified in numerical schemes \cite{formaggia2002numerical,richards2011appropriate}.   Other examples of application domains
  for the results of the present paper can be found in the introduction of \cite{GI}.
  
\subsection{A model for PDEs with incomplete data}
\label{SS:thepde}
\
This paper is concerned with recovering a function $u$ that is known to solve a specific PDE while lacking some of the data   that  would determine $u$ uniquely.   Specifically, we consider the model elliptic problem
\be 
\label{model}
-\Delta u = f \qquad \textrm{in } \Omega, \qquad u=g \qquad \textrm{on }\Gamma:=\partial \Omega,
\ee 
where $\Omega \subset \mathbb R^d$ is a bounded Lipschitz domain with  $d=2$ or $3$. 
The Lax-Milgram theorem \cite{yosida} implies the existence and uniqueness of a solution $u$ 
from the Sobolev space  $H^1(\Omega)$ to \eref{model}, once $f$ and $g$ are prescribed 
in $H^{-1}(\Omega)$ (the dual of $H^1_0(\Omega)$)  and in $H^{1/2}(\Gamma)$ (the image
of $H^1(\Omega)$ by the trace operator), respectively.

  Recall that the trace operator $T$ is defined on a function 
  $w\in C(\bar\Omega)$
  as the restriction of $w$ to $\Gamma$ 
  and this definition is then generalized to functions in Sobolev spaces by a denseness argument. 
In particular, the trace operator is  well defined on $H^1(\Omega)$. For any function $v$ in $H^1(\Omega)$ we denote by $v_\Gamma$ its trace,
\be\label{vB}
v_\Gamma:=T(v)=v|_{\Gamma},\quad v\in H^1(\Omega).
\ee 
The Lax-Milgram analysis also yields the inequalities
\be 
\label{LMbound}
c_0 \| v \|_{H^1(\Omega)} \leq \| \Delta v \|_{H^{-1}(\Omega)} + \| v_\Gamma \|_{H^{1/2}(\Gamma)}
\leq c_1 \| v \|_{H^1(\Omega)} , \qquad  v \in H^1(\Omega).
 \ee
  Here the constants $c_0,c_1$ depend on $\Omega$ and on the particular choice of norms employed on $H^1(\Omega)$ and  $H^{1/2}(\Gamma)$. 
  

Our interest centers on the question of how well we can numerically recover $u$ in the $H^1$ norm when we do not have  sufficient knowledge to guarantee a unique
solution to \eref{model}.  There are many possible settings  to which our techniques   apply, but  we shall focus  on the following scenario: 
\begin{enumerate}
\item  We  have a complete knowledge of $f$ but we do not know $g$. 
\item 
 The function  $g$ belongs to a known compact subset  
$K^B$ of $H^{\frac{1}{2}}(\Gamma)$.  Thus, membership in $K^B$ describes
our  knowledge of the boundary data.  
The function $u$ we wish to recover comes from  the set 
\be\label{classK}
K:=\{u: \ u \ {\rm solves}\ \eref{model}\ {\rm for \ some} \ g\in K^B\},
\ee 
which is easily seen from \eref{LMbound} to be a compact subset of $H^1(\Omega)$.
\item
We have access to finitely many data observations of the unknown solution $u$, in terms
of a vector
\be 
\label{multilambda}
\lambda(u):= (\lambda_1(u),\dots,\lambda_m(u))\in \R^m,
\ee 
where the $\lambda_j$ are fixed and known linear functionals defined on the functions from $K$. 
\end{enumerate}

  
  Natural candidates for the compact set $K^B$ are balls of 
Sobolev spaces 
  that are compactly embedded in $H^{\frac{1}{2}}(\Gamma)$. We thus
restrict our attention for the remainder of this paper to the case 
   \be 
   \label{special}
   K^B:=U(H^{s}(\Gamma)),
   \ee 
  for some $s>\frac 1 2$, 
  where $U(H^s(\Gamma))$ denotes the unit ball of a $H^s(\Gamma)$   with respect to the norm $\|\cdot\|_{H^s(\Gamma)}$. The precise definition of $H^{s}(\Gamma)$
  and its norm $\|\cdot\|_{H^{s}(\Gamma)}$ is described later.   We have assumed that $K^B$ is the unit ball only for convenience.
     The   arguments given below
     hold in the case when  $K^B$ is a ball of  $H^s(\Gamma)$ centered at $0$  of any radius $R$. The numerical algorithms proposed and analyzed do not require the knowledge of the radius $R$  of that ball. 
     
\subsection{The optimal recovery benchmark}

 Let $w_j:=\lambda_j(u)$, $j=1, \ldots,m$, and 
\be  
w:=(w_1,\dots,w_m)=\lambda(u)\in\R^m,
\ee 
be the vector of  data observations.
Therefore, the  totality of information we have about $u$ is that it lies in the compact set
\be 
\label{Kw}
K_w:=\{u\in K:\ \lambda(u)=w\}.
\ee

Our problem is to numerically find a function $\hat u\in H^1(\Omega)$ which approximates simultaneously all the $u\in K_w$.  This is a special case of the  problem of optimal
recovery from data (see \cite{MiR,TW,NW}). 
The optimal recovery, i.e. the best choice for $\hat u$, has the following well known theoretical description. Let $B(K_w)$ be a smallest ball in $H^1(\Omega)$ which contains $K_w$, let $R(K_w):=R(K_w)_{H^1(\Omega)}$ be its radius,
and $\hat u(w)$ its center. These are called the {\it Chebychev ball, radius, and center}, respectively.
Note that the Chebyshev ball $B(K_w)$ is unique in $H^1(\Omega)$, see \cite{devore2016data}. 
Then $R(K_w)$ is the {\em optimal recovery error}, that is, the smallest error we can have for recovering $u$ in the norm of $H^1(\Omega)$,
and $\hat u(w)$ is an {\em optimal recovery} of $u$. 

We are interested in understanding how small $R(K_w)$ is and what are the numerical algorithms 
which are {\it near optimal} in recovering $u$ from the given data $w$.  We say that an  algorithm $w\mapsto \hat u=\hat u(w)$
delivers {\it near optimal recovery} with constant $C$ if  
\be 
\label{nopt}
\|u-\hat u(w)\|_{H^1(\Omega)}\le CR(K_w), \quad w\in\R^m.
\ee 
Of course, we want  $C$ to be a reasonable constant independent of $m$. 
Our results actually deliver a recovery estimate of the form
\be 
\label{nopteps}
\|u-\hat u(w)\|_{H^1(\Omega)}\le R(K_w)+\e, \quad w\in\R^m,
\ee 
where $\e>0$ can made arbitrarily small at the price of higher computational cost.
In this sense, the recovery is near optimal with constant $C>1$ in \iref{nopt} that can be made arbitrarily close to $1$.

\subsection{A connection with the recovery of harmonic functions}

\label{SS:harmonic}
 There is a   natural restatement of our recovery problem in terms of harmonic functions.  Let $f$ be the right side of \eref{model},   where  $f$ is a known  fixed element of $H^{-1}(\Omega)$.
Let $u_0$ be the function in $H^1(\Omega)$ which is the solution to \eref{model} with $g=0$. Then, we can write any function $u\in K$ as
\be 
\label{harmonic}
u=u_0+u_\cH,
\ee 
where $u_\cH$ is a harmonic function in $H^1(\Omega)$ which has boundary value $g=T(u_\cH)$ with $g\in K^B$. Recall our assumption  that $K^B$ is the unit ball of $H^{s}(\Gamma)$ with $s>\frac{1}{2}$.

Let  
$\cH^s(\Omega)$
denote the set of harmonic functions $v$ defined on $\Omega$ for which $v_\Gamma\in H^{s}(\Gamma)$. We refer the reader to \cite{A}, where a detailed study of spaces like $\cH^s(\Omega)$ is presented. We define the norm on $\cH^s(\Omega)$ to be  the one  induced by the norm on $H^{s}(\Gamma)$, namely,
  \be 
  \label{Hnorm}
  \|v\|_{\cH^s(\Omega)}:= \|v_\Gamma\|_{H^{s}(\Gamma)},\quad v\in \cH^s(\Omega).
  \ee
There exist several equivalent definitions of  norms on $H^{s}(\Gamma)$, as discussed later.
For the moment, observe that from \iref{LMbound} it follows the existence of a constant $C_s$ such that
  \be
  \label{Hnorm1}
  \|v\|_{H^1(\Omega)}\leq  C_s\|v\|_{\cH^s(\Omega)}, \quad v\in \cH^s(\Omega).
  \ee
  Indeed, the space $\cH^s(\Omega)$ is a Hilbert space that is compactly embedded in $H^1(\Omega)$,
  as a consequence of the compact embedding of $H^s(\Gamma)$ in $H^{1/2}(\Gamma)$.
    We denote by $K^\cH$  the unit ball of $\cH^s(\Omega)$, 
  \be 
  \label{unitball}
  K^\cH:=U(\cH^s(\Omega)).
  \ee 
Since the function $u_0$ in \eref{harmonic} is fixed, it follows from \eref{special} that
\be 
\label{Hradius}
R(K_w) = R(K_{w'}^\cH)_{ H^1(\Omega)},
\quad w':=\lambda(u_\cH)=w-\lambda(u_0).
\ee

There are two conclusions that can be garnered from this reformulation.  The first is that  the optimal error in recovering  $u\in K_w$ is the same as that in  recovering the harmonic function
$u_\cH\in K^\cH_{w'}$ in the $H^1(\Omega)$ norm. The harmonic recovery  problem does not involve $f$ except in determining $w'$.  The second point is that one possible  numerical algorithm for our original problem is to first construct
a sufficiently accurate approximation $\hat u_0$ to $u_0$ and then to numerically implement an  optimal recovery of a harmonic function in $K^\cH$ from  data observations.
This numerical approach requires the computation of $w'$. In theory, $u_0$ is known to us since we have a complete knowledge of $f$.
However,  $u_0$ must be computed and any approximation $\hat u_0$  will induce an error.   Although this error  can be made arbitrarily small, it means that we only know $w'$ up to a certain numerical accuracy.  One can thus view the harmonic reformulation as an optimal recovery problem with   perturbed observations of $w'$.  The numerical
algorithm presented here  follows this approach. Its central constituent, namely the recovery of harmonic functions from a finite number of noisy observations, can be readily employed 
as well in a number of different application scenarios described e.g. in \cite{GI}.

\subsection{Optimal recovery from centrally symmetric sets in a Hilbert space}
\label{ssec:optrec}

  As noted above, the problem of recovering $u\in K_w$ is directly related to the problem of recovering the harmonic component $u_\cH\in K^\cH$ from the given data observations $w'$.  Note that  $K^\cH$ is the unit ball of the Hilbert space $\cH^s(\Omega)$.
  There is a general approach for optimal recovery from    data observations   in this Hilbert space setting, as discussed e.g. in \cite{MiR,MeM}. 
  We first recollect the general principles of this technique which will be applied in Section~\ref{SS:minmorm}    to our specific setting.
   
    Let $\cH$ be any Hilbert space and suppose that $\lambda_1,\dots,\lambda_m\in \cH^*$ are linearly independent functionals from $\cH^*$.
  Let $X$ be a Banach space such that $\cH$ is continuously embedded in $X$. 
  We are interested in optimal recovery of a function $v$ in the norm $\|\cdot\|_X$, knowing  that $v\in \cK:=U(\cH)$, the unit ball of $\cH$.
  If $w\in \R^m$ is the vector of  observations, we define
  the {\it minimal norm interpolant} as
  \be
  \label{minnorm}
  v^*(w)={\rm argmin}\{ \|v\|_\cH \; : v\in\cH\ {\rm and} \  \; \lambda(v)=w\}.
 \ee 
\begin{remark}
    \label{R:anyX}
  If $\mathcal K_w:=\{u\in \mathcal K:\ \lambda(u)=w\}$ is non-empty, the minimum norm interpolant  $v^*(w)$ is a Chebyshev center of
  $\mathcal K_w$ in $X$. That is, the minimal norm interpolant gives optimal recovery with constant $C=1$.   In other words, this Chebyshev center does not depend on $X$.  The radius $R(K_w)_X$ will however generally depend on $X$.
\end{remark}
  To prove this remark, first note that any $v\in \mathcal K_w$ may be written as 
  $v=v^*(w)+\eta $ where $\eta$ belongs to  the null space $\cN$ of $\lambda$.   Because $v^*(w)$ has minimal norm,  $\eta$ is orthogonal to $v^*(w)$  and hence from the Pythagorean theorem
  $$
  \|v-v^*(w)\|_{\cH}^2=  \|v\|_{\cH}^2-\|v^*(w)\|_{\cH}^2\le  1-\|v^*(w)\|_{\cH}^2=:r^2,
  $$
   because $\|v\|_{\cH}\le 1$.   Notice that $v^*(w)-\eta$ is also in $\mathcal K_w$. It follows that  $\mathcal K_w$ is precisely  the ball in the affine space $v^*(w)+\cN$
  centered at $v^*(w)$ and of radius $r$. In particular, $\mathcal K_w$   is centrally symmetric around $v^*(w)$ and we now show that $v^*(w)$ is a Chebyshev center of a Chebyshev ball in $X$. Recall that unless $X$ is uniformly convex, these quantities might not be uniquely defined \cite{devore2016data}.  Let $u \in \mathcal K_w$ be the furthest away from $v^*(w)$, i.e., 
  $$
u \in \arg\max_{v \in  \mathcal K_w} \| v^*(w) - v \|_X
$$
and  set $\delta :=  \| v^*(w) - u \|_X$.
Since $\mathcal K_w$ is centrally symmetric around $v^*(w)$, $u' := 2v^*(w)-u \in K_w$ and $\| u - u'\|_X = 2\delta$. 
For any $\bar v \in  X$ with $\bar v \not  = v^*(w)$, the triangle inequality yields
$$
2 \delta = \| u - u'\|_X \leq \| u - \bar v\|_X + \| u'-\bar v\|_X \leq 2 \max_{v \in \mathcal K_w} \| \bar v - v \|_X,
$$
and so
$$
\max_{v \in  \mathcal K_w} \| v^*(w) - v \|_X \leq \max_{v \in \mathcal K_w} \| \bar v - v \|_X,
$$
which shows that $v^*(w)$ is indeed a Chebyshev center in $X$.
In particular
$$
 \|v-v^*(w)\|_X \leq R(\cK_w)_X, \quad v\in \mathcal K_w.
  $$

Standard Hilbert space analysis shows that the mapping $w\mapsto v^*(w)$
is a linear operator. More importantly, it has a natural expression
that is useful for numerical computation.  Namely, from the Riesz representation
  theorem each $\lambda_j$ can be described as
  \be 
 \nonumber
  \lambda_j(v)= \langle v,\phi_j\rangle_{\cH}, \quad v\in \cH,
  \ee 
  where   $\phi_j\in\cH$ is called the Riesz representer of $\lambda_j$. The minimal norm interpolant
  has the representation
  \be 
  \label{v*rep}
  v^*=\sum_{j=1}^m a^*_j\phi_j, 
  \ee 
  where
  $a^*=(a^*_1, \ldots,a^*_m)$ solves the system of equations
  \be 
\nonumber
  Ga^*=w,\quad G:=(\langle \phi_i,\phi_j\rangle_{\cH})_{i,j=1,\dots,m},
\ee
with  $G$ being the Gramian matrix associated to $\phi_1,\dots,\phi_m$. 

\begin{remark}
\label{remproj}
Note that $v^*(w)$ is exactly the $\cH$-orthogonal projection of $u$
onto the space spanned by the $m$ Riesz representers.
\end{remark}

\begin{remark}
In the case where $\cH$ is a more general Banach space, we are still ensured
 that the minimal norm interpolation is a near-optimal recovery with constant $C=2$.
 However, its dependence on the data $w$ is no longer linear and the above
observation regarding its computation
does not apply.
 \end{remark}

Our proposed numerical recovery scheme is based on approximately realizing
\eref{v*rep}.

\subsection{Objectives and outline}
\
The    main goal of this paper 
is to create numerical algorithms which are guaranteed to produce  a near optimal recovery   $\hat u$   from the given data $w$ and to analyze their practical implementation. 
 We begin in \S \ref{S:hs} with some remarks on the definition of the space $H^s(\Gamma)$ and  its norm, which are of importance  both in the accuracy analysis and the practical implementation of
recovery algorithms.

Next we turn to the description of our numerical algorithms.  The algorithms we propose and analyze are based on the  general approach for optimal recovery  described  in Section~\ref{ssec:optrec} when this approach is applied to our particular PDE setting.
We describe a solution algorithm which takes into consideration the effect of numerical
perturbations.  We  first consider the case when the linear functionals $\lambda_j$ are defined on all of $H^1(\Omega)$ 
and then adapt this algorithm to the case when the linear functionals are point evaluations 
\be 
\label{pe}
\lambda_j(u):=u(x_j),\quad 
x_j\in \overline\Omega,\quad j=1, \ldots,m.
\ee 
Point evaluations are not defined on all of $H^1(\Omega)$ when $d>1$,  however, they are defined on $K$ when the smoothness order $s$ is large enough.

The critical ingredient in our proposed algorithm is the numerical computation of the
Riesz representers $\phi_j$ of the restrictions of  $\lambda_j$ 
to the Hilbert space $\cH^s(\Omega)$. 
Each of these Riesz representers  is characterized as a solution
  to an elliptic problem and can be computed
 offline since it does not involve the 
measurement vector $w$.  Our suggested numerical method for approximating $\phi_j$ is  based on finite element discretizations and is discussed in \S \ref{S:numerical}.   We establish
quantitative error bounds  for the numerical approximation in terms of the mesh size.  
Numerical illustrations of the  optimal recovery algorithm are given in \S\ref{SS:Rr}.

Note that the optimal recovery error over the class $K$
 strongly depends on the choice of the linear functionals $\lambda_j$.
For example, in the case of point evaluation, this error can be very large if the data sites $\{x_j\}_{j=1}^m$ are poorly positioned, or small if they are optimally positioned. This points to the importance  of the Gelfand widths
and sampling numbers.  They describe the optimal recovery error over $K$ with optimal choice of  functionals 
in the general case and the point evaluation case, respectively. The numerical behaviour of these quantities 
in our specific setting is 
discussed in \S \ref{S:Gelfandwidths}.

\section{The spaces $H^s(\Gamma)$ and $\cH^s(\Omega)$}
\label{S:hs}

In this section, we discuss the definition and basic properties of the spaces 
 $H^{s}(\Gamma)$ and $\cH^s(\Omega)$.  We refer to \cite{Ad} for a general treatment of Sobolev spaces
on domains $D\subset \R^d$. Recall that for fractional orders $r>0$, the norm of $H^r(D)$ is defined as
$$
\|v\|_{H^r(D)}^2:=\|v\|_{H^k(D)}^2+\sum_{|\alpha|=k}
\,\int_{D\times D}\frac {|\partial^\alpha v(x)-\partial^\alpha v(y)|^2}{|x-y|^{d+2(r-k)}}dxdy,
$$
where $k$ is the integer such that $k<r<k+1$,
and $\|v\|_{H^k(D)}^2:=\sum_{|\alpha|\leq k}\|\partial^\alpha v\|_{L^2(D)}^2$ is the standard $H^k$-norm.

\subsection {Equivalent definitions of $H^s(\Gamma)$}
\label{SS:HsGamma}

Let $\Omega$ be any bounded Lipschitz domain in $\R^d$.  We recall the trace operator $T$
introduced in \S\ref{SS:thepde}.  One first possible definition of
the space $H^{s}(\Gamma)$, for any 
 $s\geq\frac{1}{2}$, is as the restriction of $H^{s+\frac 1 2}(\Omega)$ to $\Gamma$, that is,
 $$
 H^{s}(\Gamma)=T(H^{s+\frac 1 2}(\Omega)),
 $$
 with norm
\be
\|g\|_{H^{s}(\Gamma)}:=\min \Big\{\|v\|_{H^{s+\frac 1 2}(\Omega)} \; : \; v_\Gamma=g\Big\}.
\label{sobnorm1}
\ee
The resulting norm is referred to as the trace norm definition for $H^s(\Gamma)$.

There is a second, more intrinsic way to define  $H^{s}(\Gamma)$, by properly adapting the notion of Sobolev smoothness to the boundary. This can be done by locally mapping 
the boundary onto domains of $\R^{d-1}$ and requiring that the pullback of $g$ by
such transformation have $H^s$ smoothness on such domains. We refer the reader to \cite{Gr} and \cite{Nec} for the
complete intrinsic definition, where it is proved to be equivalent to the trace definition
for a range of $s$ that depends on the smoothness of the boundary $\Gamma$.

For small values of $s$, Sobolev norms 
for $H^s(\Gamma)$ may also be equivalently defined without the help of
local parameterizations, as contour integrals. For example, if $0<s<1$ and $\Omega$ is a Lipschitz domain,  
we define
$$
\|g\|_{H^s(\Gamma)}^2:=\|g\|_{L^2(\Gamma)}^2+\int_{\Gamma\times \Gamma}\frac {|g(x)- g(y)|^2}{|x-y|^{d-1+2s}}dxdy,
$$
and if $s=1$ and $\Omega$ is a polygonal domain, we define
\be 
\label{eqH1}
\|g\|_{H^1(\Gamma)}^2:=\|g\|_{L^2(\Gamma)}^2+
\|\nabla_{\Gamma} g\|_{L^2(\Gamma)}^2,
\ee  
where $\nabla_{\Gamma}$ is the tangential gradient, and likewise 
$$
\|g\|_{H^s(\Gamma)}^2:=\|g\|_{H^1(\Gamma)}^2+\int_{\Gamma\times \Gamma}\frac {|\nabla_{\Gamma} g(x)- \nabla_{\Gamma} g(y)|^2}{|x-y|^{d-1+2(s-1)}}dxdy,
$$
for $1<s<2$. In the numerical illustration given in \S\ref{SS:Rr}, we will specifically take the value $s=1$
and a square domain, using the definition \eref {eqH1}.

When $\Omega$ has smooth boundary, it is known that the trace definition and intrinsic definition of 
the $H^s(\Gamma)$ norms are equivalent for all $s\geq 1/2$. 
On the other hand, when $\Omega$ does not have a smooth boundary, it is easily seen that 
the two definition
are not equivalent unless restrictions are made on $s$. Consider for example
the case of polygonal domains of $\R^2$: it is easily seen 
that the trace $v_\Gamma$ of a smooth function $v\in C^\infty (\o \Omega)$ 
has a tangential gradient $\nabla_{\Gamma}v_\Gamma$ that generally has jump discontinuities
at the corner points and thus does not belong to $H^{1/2}(\Gamma)$. In turn, the 
equivalence between the trace and intrinsic norms only holds for $s<\frac 3 2$ and 
in such case we limit the value of $s$ to this range. The same restriction $s<3/2$ 
applies to a polyhedral domain in the case $d=3$.

\subsection {The regularity of functions in $\cH^s(\Omega)$}
\label{SS:regHs}

We next give some remarks on the Sobolev smoothness of functions
from the space $\cH^s(\Omega)$ when $s>1/2$. Clearly such harmonic functions are infinitely smooth
inside $\Omega$ and also belong to $H^1(\Omega)$, but one would like to know for which value of $r$
they belong to $H^r(\Omega)$. To answer this question, we consider $v\in \cH^s(\Omega)$.  By the  definition of $\cH^s(\Omega)$, $v$ is harmonic in $\Omega$ and 
 $v_\Gamma\in H^s(\Gamma)$. Having assumed that $s$ in the admissible
range where all above definitions of the $H^s(\Gamma)$ norms are equivalent, and using 
the first one, we know that there exists a function
$\t v\in H^{s+\frac 1 2}(\Omega)$ such that $\t v_\Gamma=v_\Gamma$
$$
\|\t v\|_{H^{s+\frac 1 2}(\Omega)}= \|v_\Gamma\|_{H^s(\Gamma)}=\|v\|_{\cH^s(\Omega)}.
$$
We define $\o v:=v-\t v$ so that $v=\t v +\o v$. We are interested in the regularity of $\o v$ since it will give the regularity of $v$.   Notice that   $\o v_\Gamma=0$ and 
$$
-\Delta \o v=\o f:=\Delta \t v.
$$
 The function $\o f$ belongs to the Sobolev space $H^{s-\frac 3 2}(\Omega)$ and we are left with
the classical question of the regularizing effect in Sobolev scales when solving the Laplace equation
with Dirichlet boundary conditions. Obviously, when $\Omega$ is smooth, we find that $\o v\in H^{s+\frac 1 2}(\Omega)$
and so we have obtained the continuous embedding
$$
\cH^s(\Omega)\subset H^r(\Omega), \quad r=s+\frac 1 2.
$$
For less smooth domains, the smoothing effect is limited (in particular by the presence of singularities on the boundary of $\Omega$),  i.e., $\o v $ is only guaranteed to be in $H^r(\Omega)$ where
$r$ may be less than $s+1/2$,
see \cite{Gr}.  More precisely
$$
\cH^s(\Omega)\subset H^r(\Omega),
$$
where
\be  
\label{definer} \quad r:=\min\Big\{s+\frac 1 2,r^*\Big\},
\ee  
Here, $r^*=r^*(\Omega)$ is the limiting bound for the smoothing effect:
\begin{enumerate}
\item
For smooth domains $r^*=\infty$.
\item
For convex domains $r^*=2$.
\item
For non-convex polygonal domains in $\R^2$, or a polyhedron in $\R^3$, one has $3/2<r^*<2$ where
the value of $r^*$ depends on the reentrant angles. 
\item
In particular for polygons, we can take $r^*=1+\frac \pi\omega-\e$, for any $\e>0$ where $\omega$ is the largest inner angle.
\end{enumerate}
Note that $r^*$ could be strictly smaller than 
$s+\frac 1 2$. 

In summary, 
for an admissible range of $r>1$ that depends on $s$ and $\Omega$
one has the continuous embedding $\cH^s(\Omega)\subset H^r(\Omega)$, and so there exists
a constant $C_1$ that depends on $(r,s)$ and $\Omega$, such that
\be
\|v\|_{H^r(\Omega)}\leq C_1\|v\|_{\cH^s(\Omega)} = C_1\|v_\Gamma\|_{H^s(\Gamma)}, \quad v\in \cH^s(\Omega).
\label{hshr}
\ee

\section{A near optimal recovery algorithm}
\label{S: firstnumericalalgorithm}

In this section, we present  a  numerical algorithm for solving \eref{model} when the information about the boundary value $g$ is  
incomplete. 
We first work under the assumption that the $\lambda_j$'s are  continuous over $H^1(\Omega)$,
and assumed to be  linearly independent (linear independence can be guaranteed by throwing away dependent functionals when necessary). We prove that the proposed  numerical  recovery algorithm is near optimal. We then adapt our approach to the case where the
$\lambda_j$'s are point  evaluations, see \eref{pe}, and therefore not continuous over $H^1(\Omega)$ when $d\geq 2$.

  \subsection{Minimum norm data fitting}
  \label{SS:minmorm} 
\
 
 We now apply the general optimal recovery principles discussed in Section~\ref{ssec:optrec}  to our specific setting in which 
$$
\cH = \cH^s(\Omega) \qquad \textrm{and} \qquad  X=H^1(\Omega).
$$

Let $\phi_j\in \cH^s(\Omega)$ be the Riesz representer of the functional $\lambda_j$ when viewed as a functional on $\cH^s(\Omega)$.  In other words
  \be 
\nonumber
  \lambda_j(v)=\langle v,\phi_j\rangle_{\cH^s(\Omega)},\quad v\in \cH^s(\Omega).
  \ee 
 We assume that  the $\lambda_j$ are linearly independent   on $\cH^s(\Omega)$ and thus  the Gramian matrix 
  \be 
 \nonumber
  G=\(g_{i,j}\)_{i,j=1,\dots,m},\quad g_{i,j}:=\langle \phi_i,\phi_j\rangle_{\cH^s}= \lambda_j(\phi_i),
  \ee 
  is invertible. 
  
  Now, let $u=u_0+u_\cH$, with $u_\cH\in K^\cH=U(\cH^s(\Omega))$ be the function in $K$ that gave rise to our data observation $w$.  So, we have  
  \be
  \nonumber
  w'=w-\lambda(u_0)=\lambda(u_\cH).
  \ee  
  If  $a^*$ is the vector in $\R^m$ which satisfies $Ga^*=w'$,  then $u_\cH^*:=\sum_{j=1}^ma_j^*\phi_j$ is the function 
  of minimum $\cH^s(\Omega)$ norm which satisfies the data $w'$, i.e., $\lambda(u_\cH^*)=w'$.   We have seen that
  \be
  \nonumber
  \|u_\cH-u_\cH^*\|_{H^1(\Omega)}\leq R(K^\cH_{w'})_{H^1(\Omega)},
  \ee
  namely, $u_\cH^*$ is the  optimal recovery  of the functions in $K^\cH_{w'}$.  Note that the recovery error is measured in $H^1$ not in 
  $\cH^s(\Omega)$.  In turn, see \eref{Hradius}, the function  $u^*:=u_\cH^*+u_0$ is  the  optimal recovery for functions in $K_w$:
  \be
  \nonumber
 \|u-u^*\|_{H^1(\Omega)} \leq R(K_{w})_{H^1(\Omega)}.
  \ee

   The idea behind our proposed  numerical method is to  numerically construct a function  $\hat u\in H^1$ that approximates $u^*$ well.  If, for example, we have for $\e>0$ the bound
  \be  
 \nonumber
  \|u^*-\hat u\|_{H^1(\Omega)}\le \e ,
  \ee 
  then for any $u\in K$, we have by the triangle inequality
   \be 
 \nonumber
   \|u-\hat u\|_{H^1(\Omega)} \le R(K_w)_{H^1(\Omega)}+ \e. 
   \ee 
   Given any  $C>1$, by taking $\e$ small enough, we have  that $\hat u$ is a near best recovery 
   of the functions in $K_w$ with constant $C$.
   
  \subsection{The numerical recovery algorithm for $H^1$-continuous functionals}\label{ss:numerical_algorithm}
  
  Motivated by the above analysis, we propose the following numerical algorithm for solving our recovery problem. 
  The algorithm involves approximations of the function $u_0$ and the Riesz representers $\phi_j$, typically computed by finite element discretizations,
  and the application of the linear functionals $\lambda_j$ to these approximations. In order to avoid extra technicalities, 
  we make here the assumption that the applications of the functionals to a known finite element function can be exactly computed. 

 We first work under the additional assumption that the linear functionals $\lambda_j$
are not only defined on $K$ but that they are  continuous over $H^1(\Omega)$.  We define $\Lambda$ as the maximum of the norms of the $\lambda_j$ on $H^1(\Omega)$. In this case 
  \be 
  \label{fhs} 
  |\lambda_j(v)|\le \Lambda \|v\|_{H^1(\Omega)},\quad v\in
  H^1(\Omega).
  \ee 
In what follows, throughout this paper, we use the following weighted $\ell_2$ norm  on $\R^m$,
\be 
\nonumber
\|z\|:=\left(\frac{1}{m}\sum_{j=1}^m|z_j|^2\right)^{1/2}=m^{-1/2}\|z\|_{\ell^2},\quad z=(z_1,\ldots,z_m)\in\R^m.
\ee
In particular, we have
\be 
\label{normone}
\nonumber
\|\lambda(v)\|\le 
 \Lambda\|v\|_{H^1(\Omega)}, \quad v\in H^1(\Omega).
\ee

  Given a user prescribed accuracy $\e>0$, our algorithm does the following four steps involving intermediate tolerances {$(\e_1,\e_2)$.
  \vskip .1in
  \noindent{\underline{\bf Step 1:} {\it  We numerically find an approximation $\hat u_0$ to $u_0$ which satisfies
  \be 
  \label{S1}
  \|u_0-\hat u_0\|_{H^1(\Omega)}\le \e_1. 
  \ee 
  }
  
  \vskip .1in
  \noindent  
  To find such a $\hat u_0$, we   use standard or adaptive FEM methods.
   Given that $\hat u_0$ has been constructed, we define 
   $\hat w:=w-\lambda(\hat u_0)$. 
   Then, for $w':=w-\lambda(u_0)$ we have, see \eref{normone}, 
  \be 
  \label{s111}
  \|w'-\hat  w\|\le  \Lambda\e_1.
  \ee
  On the other hand, since
  $|\lambda_j(v)|\leq  \Lambda\|v\|_{H^1(\Omega)}\leq  \Lambda_s\|v\|_{\cH^s(\Omega)}\le  \Lambda_s$, 
  where
  $$
  \Lambda_s:=C_s\Lambda,
  $$
  see 
   \eref{fhs}, \eref{Hnorm1} and \eref{unitball}, we derive that
 \be
 \label{boundnormw}
 \|w'\|\le \Lambda_s.
 \ee 
Thus by triangle inequality, we also find that
  \be 
  \label{boundwhat}
  \|\hat  w\|\le \Lambda_s+\Lambda \e_1.
  \ee 
  \vskip .1in
  \noindent
  \underline{\bf Step 2:}  {\it  For each $j=1,\dots,m$, we numerically compute an approximation $\hat \phi_j\in H^1(\Omega)$  to $\phi_j$ which satisfies
  \be 
  \label{hatphij}
  \|\phi_j-\hat \phi_j\|_{H^1(\Omega)} \le \e_2,\quad j=1,\dots,m.
  \ee }
  \vskip .1in
  \noindent
This numerical computation is  crucial  and is performed during the  offline phase of the algorithm. We detail it
in \S \ref{S:numerical}. Note that the $\hat\phi_j$'s are not assumed to be in $\cH^s(\Omega)$, and in particular not assumed to be harmonic functions.
    
  \vskip .1in
  \noindent
  \underline{\bf Step 3:} {\it We define and compute the   matrix 
  \be
  \nonumber
  \hat G=(\hat g_{i,j})_{i,j=1,\dots,m},
  \quad \hat g_{i,j}:= \lambda_j(\hat \phi_i),
  \ee
 and thus  $|\hat g_{i,j}- g_{i,j}|\leq \Lambda\e_2$ for all $i,j$.}
  \vskip .1in
  \noindent

 It follows that for  the matrix $R:=G-\hat G$
 we have
  \be
  \nonumber
  \|R\|_{1}\le  m\Lambda\e_2,
  \ee
where we use the shorthand notation
 $\|\cdot\|_1:=\|\cdot\|_{\ell_1\to \ell_1}$ for matrices.
  Since $G$ is invertible, we are ensured that $\hat G$ is also  invertible   for $\e_2$ small enough.
  We define 
  \be
  \nonumber
  M:=\|G^{-1}\|_{1}, \quad \hat M:=\|\hat G^{-1}\|_{1}.
  \ee
  While these two norms are finite, their size will depend on the nature and the positioning of the linear functionals $\lambda_j$, $j=1,\dots,m$, as it will be seen in the section on numerical experiments.
These two numbers are close to one another when $\e_2$ is small since   $\hat M$  converges towards the unknown quantity $M$ as $\e_2\to 0$.  
  In particular, we have 
  $$
  |M-\hat M|=|\|G^{-1}\|_1-\|\hat G^{-1}\|_1|\leq \|G^{-1}-\hat G^{-1}\|_1=\|\hat G^{-1}RG^{-1}\|_1\leq M\hat Mm\Lambda\e_2,
  $$
  from which we obtain that
  \be
  \label{pp1}
M\leq \frac{\hat M}{1-m\hat M \Lambda\e_2} \quad {\rm and}
\quad \hat M\leq \frac{M}{1-mM \Lambda\e_2},
  \ee
 provided that $mM\Lambda\e_2<1$ and $m\hat M\Lambda\e_2<1$.  We also have the bound
 \be 
 \label{Gerror}
 \|\hat G^{-1}-G^{-1}\|_{1}\leq \frac {M^2}{1-mM\Lambda\e_2} m\Lambda\e_2=:\delta. 
 \ee 
 It is important to observe that $\delta$ can be made arbitrarily small by diminishing $\e_2$.

  \vskip .1in
  \noindent
  \underline{\bf Step 4:} {\it  We numerically solve the $m\times m$
  algebraic system $\hat G \hat a= \hat w$, thereby finding a vector $\hat a=(\hat a_1,\dots,\hat a_m)$.
  We then define $\hat u_\cH:=\sum_{j=1}^m\hat a_j\hat \phi_j$ and
  our recovery of $u$ is  $\hat u:=\hat u_0+\hat u_\cH$.}
  
  \vskip .1in
  \noindent 
  This step can be implemented by standard linear algebra solvers.  
\vskip .1in

   One major advantage of the above algorithm
   is that Steps 1-2-3 can be performed offline since they do not involve the data $w$.  That is, we can
  compute $\hat u_0$, the approximate Riesz representers $\hat \phi_j$ and the approximate Gramian $\hat G$ and its inverse without knowing $w$.  
  In this way, the computation of $\hat u$ from given data $w$ can be done fast online
  by Step 4 which only involves solving an $m\times m$ linear system.
  This may be a significant advantage, for example, when having to process a large number of measurements for the same set of sensors.

  \subsection{A near optimal recovery bound}

   The following theorem shows that a near optimal recovery of $u$ can be reached provided that the tolerances in the above described algorithm are chosen small enough.
  
  \begin{theorem} 
  \label{T:H2}
  For any prescribed $\e>0$, if the tolerances $(\e_1,\e_2)$,
 are small enough such that
   $mM\Lambda\e_2<1$ and
  \be 
  \label{eps}
\e_1+mM \Lambda_s\e_2+( C_0+\e_2)(mM\Lambda\e_1 +m(\Lambda_s+\Lambda\e_1 )\delta)
  \leq \e, 
  \ee 
where $C_0:=\max_{j=1,\ldots,m}\|\phi_j\|_{H^1(\Omega)}$ and 
$\delta:= \frac {M^2}{1-mM\Lambda\e_2} m\Lambda\e_2$, then
the function $\hat u$ generated by the above algorithm
  satisfies 
  \be  
\nonumber
  \|u-\hat u\|_{H^1(\Omega)} \le    R(K_w)_{H^1(\Omega)}+\e,\quad \hbox{for every}\quad u\in K_w.
  \ee 
  Thus, for any $C>1$ it is a  near optimal recovery of
  $u$ with constant $C$ provided $\e$ is taken sufficiently small. 
  \end{theorem}
  \begin{proof} Let  $u=u_0+v$  be our target function in $K_w$.  We define $w'=w-\lambda(u_0)$ and  $v^*:=v^*(w')$ which is the Chebyshev center of $K^\cH_{w'}$. We recall the algebraic system $Ga^*= w'$ associated to the characterization of $v^*$ (see \eref{v*rep}). We write 
  \be 
  \label{eq1}
  \|u_\cH^*-\hat u_\cH\|_{H^1(\Omega)}\leq \Big\|\sum_{j=1}^m a_j^*(\phi_j- \hat \phi_j)\Big\|_{H^1(\Omega)}+\Big\|\sum_{j=1}^m (a_j^*-\hat a_j) \hat \phi_j\Big\|
  _{H^1(\Omega)}
  \leq \|a^*\|_{\ell_1}\e_2+\|a^*-\hat a\|_{\ell_1}( C_0+\e_2),
  \ee
 where we have used \eref{hatphij} and the fact that 
  $$
  \|\hat \phi_j\|_{H^1(\Omega)}\leq
  \|\phi_j\|_{H^1(\Omega)}+\|\phi_j-\hat\phi_j\|_{H^1(\Omega)}\leq C_0+\e_2. 
  $$
Note that  
  \be
  \label{eq2}
  \|a^*\|_{\ell_1}=\|G^{-1}w'\|_{\ell_1}
  \le   M \| w'\|_{\ell_1} \leq  Mm
   \|w'\|\leq mM \Lambda_s,
  \ee
  where we have used that $\|w'\|_{\ell_1} \leq  m
  \|w'\|$
  and inequality \eref{boundnormw}. 
  Therefore it follows from \eref{eq1} and \eref{eq2} that 
  \be 
  \label{vest}
  \|u_\cH^*-\hat u_\cH\|_{H^1}\leq mM \Lambda_s\e_2+\|a^*-\hat a\|_{\ell_1}( C_0+\e_2).
  \ee
  For the estimation of $\|a^*-\hat a\|_{\ell_1}$, we introduce the intermediate vector $\tilde a\in \R^m$, which is the  solution to the system $G\tilde a=\hat w$. Clearly, 
$$  
  \|\tilde a - a^*\|_{\ell_1}
  =  \|G^{-1}(\hat w-w')\|_{\ell_1}
  \leq M \|\hat w-w'\|_{\ell_1}
  \leq M m\|\hat w-w'\|\leq mM \Lambda\e_1,
$$ 
  where we invoked \eref{s111}.  On the other  hand,  in view of \eref{Gerror}  and \eref{boundwhat}, we have
$$
  \|\tilde a -  \hat a\|_{\ell_1}
  =\|(G^{-1}-\hat G^{-1})\hat w\|_{\ell_1}
   \leq \delta \|\hat w\|_{\ell_1}
\leq m\delta \|\hat w\|
  \leq m(\Lambda_s+\Lambda\e_1 )\delta.
$$ 
  Combining these  two estimates, 
   we find that
$$ 
  \|a^*-\hat a\|_{\ell_1}
 \leq mM\Lambda\e_1 +m(\Lambda_s+\Lambda\e_1 )\delta.
$$  
 We now insert this bound into \eref{vest} to obtain 
  \be 
  \nonumber
  \|u_\cH^*-\hat u_\cH\|_{H^1(\Omega)}\leq mM\Lambda_s\e_2+(C_0+\e_2)(mM\Lambda\e_1 +m(\Lambda_s+\Lambda\e_1 )\delta).
  \ee
   Thus, for $u^*:=u_0+u_\cH^*$ and using \eref{S1}, we have
  \begin{eqnarray} 
  \nonumber
  \|u^*-\hat u\|_{H^1(\Omega)}
  &\leq &\|u_0-\hat u_0\|_{H^1(\Omega)}
  + \|u_\cH^*-\hat u_\cH\|_{H^1(\Omega)}\\
  &\leq & 
 \e_1+mM \Lambda_s\e_2+( C_0+\e_2)(mM\Lambda\e_1 +m(\Lambda_s+ \Lambda\e_1)\delta)
  \leq \e,
  \label{finalbound}
  \end{eqnarray} 
     Since $u=u_0+u_\cH$, we have

$$    
     \|u-  u^*\|_{H^1(\Omega)}=\|u_\cH-u_\cH^*\|_{H^1(\Omega)}\le R(K^\cH_{w'})_{H^1(\Omega)} = R(K_w)_{H^1(\Omega)},
$$     
    and the statement of the theorem  follows from this inequality and \eref{finalbound}.
  \end{proof}

\begin{remark}
We point out that $M$, the  norm
of the matrix $G^{-1}$, grows potentially fast as $m$ gets larger indicating that the
Riesz representers become closer to be linearly dependent. This leads to numerical
difficulties when computing the estimator which, as noted in Remark \ref{remproj},
amounts to the $\cH$-orthogonal projection onto the space spanned by
the representers. Therefore, regularization strategies might also be needed to compete with the ill-conditioning
of $G$ in the asymptotic regime.
One typical strategy is to set to $0$ the smallest eigenvalues of $G$
according to a given threshold and apply the pseudo-inverse. We refer for example
to \cite{AH} where this regularized projection strategy is reviewed and analyzed in detail.
\end{remark}

 \begin{remark}
 \label{Remark1}
Note that in numerical computations the quantity $\hat M$ is available while $M$
is unknown. Thus in practice, in order to achieve the prescribed accuracy $\e$, we can first impose that 
$\e_2<(2m\hat M\Lambda)^{-1}$ and derive the inequalities, see \eref{pp1},
$$
M\leq \frac{\hat M}{1-m\hat M\Lambda\e_2}\leq 
2\hat M,\quad \|G^{-1}-\hat G^{-1}\|_1\leq 
\frac{\hat M^2}{1-m\hat M\Lambda \e_2}m\Lambda\e_2\leq 2\hat M^2m\Lambda\e_2=:\hat \delta,
$$
where the last inequality is proven in a similar fashion to \eref{Gerror}. If we then  follow the proof of Theorem 
\ref{T:H2}, the requirement in \eref{finalbound} can be substituted by
\begin{eqnarray}
\nonumber
\e_1+2m\hat M \Lambda_s\e_2+( C_0+\e_2)(2m\hat M\Lambda\e_1 +
m(\Lambda_s+\Lambda\e_1 )\hat\delta)
  \leq \e,
  \label{eq5}
\end{eqnarray}
 and thus  all participating quantities are computable.
 \end{remark}

\begin{remark}
 \label{Remark2}
 The  result in Theorem \ref{T:H2} can easily be extended to the case of noisy data, that is, to the case when 
the observations 
\be
\nonumber
\t w=w+\eta,
\ee
where $\eta$ is a noise vector of norm $\|\eta\|\leq \kappa$. Indeed, the application of the algorithm
to this noisy data  leads to finding  in Step 1 the vector $\hat w:=w+\eta-\lambda(\hat u_0)$ that satisfies
\be
\nonumber
\|w'-\hat w\|\leq  \Lambda\e_1+\kappa, \quad  \hbox{and}\quad \|\hat w\|\leq \Lambda_s+\e_1 \Lambda+\kappa,
\ee
 where $w'=w-\lambda(u_0)$.
Inspection of the above proof shows that under the same assumption as in Theorem \ref{T:H2}, one 
has the recovery bound
{\rnew }
\be  
 \nonumber
  \|u-\hat u\|_{H^1(\Omega)} \le   R(K_w)_{H^1(\Omega)}+\e +C\kappa ,\quad \hbox{for every}\quad u\in K_w,
\ee 
 where $C:=(M+\delta)m(C_0+\e_2)$.
\end{remark}

\begin{remark}
\label{rem:4.4}
For simplicity, we did not introduce in the above analysis the possible errors 
in the application of the $\lambda_i$ to the
approximations $\hat u_0$ and $\hat \phi_j$, and in the numerical solution to the system $\hat G \hat a=\hat w$,
which would simply result in similar conditions involving the extra tolerance parameters.
\end{remark}

\subsection {Point evaluation data} 
\label{SS:point}

We now want to extend  the numerical algorithm and its analysis
to the case when the data functionals $\lambda_j$, $j=1,\dots,m$, are point evaluations
\be 
\nonumber
\lambda_j(h):=h(x_j),
\quad x_j\in \overline\Omega, \quad j=1,\dots,m.
\ee 
Of course these functionals are not defined for general functions $h$ from  $H^1(\Omega)$. However, we can  formulate the recovery problem whenever the functionals
$\lambda_j$ are well defined on $K$.   We now discuss settings when this is possible.

Recall that any $u\in K$ can be written as $u=u_0+u_\cH$, where $u_0$ is the solution to \eref{model} with right side $f$ and $g=0$ and $u_\cH\in\cH^s(\Omega)$.  Point evaluation is well defined for the harmonic functions $u_\cH\in \cH^s(\Omega)$, provided the points are in $\Omega$.  In addition, they are well defined for points on the boundary $\Gamma$ if 
the space $\cH^s(\Omega)$ continuously embeds 
into $C(\o\Omega)$.  For $d=2$, this is the case when $s>1/2$ and when $d=3$, this is the case when $s>1$.  

Concerning $u_0$,  we will need some additional assumption to guarantee that point evaluation of $u_0$ makes sense at the data sites $x_j$, $j=1, \ldots,m$. For example, it is enough to assume that  $u_0$ is globally continuous or at least in a neighborhood of each of these points.  This can be guaranteed by assuming an appropriate regularity of $f$.  In this section, we assume that one of these settings holds.
We  then write
 \be 
 \nonumber
 w_j':=u_\cH(x_j)=w_j-u_0(x_j),\quad  j=1,\dots,m,
 \ee
and follow the algorithm  of the previous section
with the following simple modifications:  
  \vskip .1in
  \noindent{{\bf Modified Step 1:} {\it  We numerically find an approximation $\hat u_0$ to $u_0$, which in addition to
  \be 
  \nonumber
  \|u_0-\hat u_0\|_{H^1(\Omega)}\le \e_1,
  \ee 
 satisfies the requirement
 \be 
  \label{S13}
  \max_{i=1,\dots,m}|u_0(x_i)-\hat u_0(x_i)|\le \e_1.
  \ee 
  }
  \vskip .1in
  \noindent  
  To find such a $\hat u_0$ we   use standard or adaptive FEM methods.
   Given that $\hat u_0$ has been constructed, we define 
   $\hat w_j:=w_j-\hat u_0(x_j)$, $j=1,\ldots,m$, and thus, using \eref{S13},  we have
  $\|w'-\hat  w\|\le \e_1$. 

  \vskip .1in
    
  \noindent{{\bf Modified Step 2:}  {\it  For each $j=1,\dots,m$, we numerically compute an approximation $\hat \phi_j$ to $\phi_j$,  which in addition to
  \be 
  \nonumber
  \|\phi_j-\hat \phi_j\|_{H^1(\Omega)} \le \e_2,\quad j=1,\dots,m,
  \ee
  satisfies the condition
  \be 
  \label{S13phi}
  \max_{i=1,\dots,m}|\phi_j(x_i)-\hat \phi_j(x_i)|\le \e_2, \quad i,j=1,\dots,m.
  \ee 
  }
  \vskip .1in
  \noindent
  Condition \eref{S13phi} ensures that in Step 3 we can choose the entries $\hat g_{i,j}$ of the matrix $\hat G$ as 
 $$
 \hat g_{i,j}=\hat \phi_j(x_i), \quad i,j=1, \ldots,m.
 $$
 The Steps 3 and 4 of our algorithm remain the same as in the previous section.
 
 \begin{theorem}
 \label{T:pointtheorem}
 
  With the above modifications,  
  Theorem \ref{T:H2} holds with the exact same statement in
  this point evaluation setting.
  \end{theorem}
   \vskip .1in
   \noindent
   \begin{proof} The proof is the same as that of Theorem \ref{T:H2}.
   \end{proof}

\section{Finite element approximations of the Riesz representers}
  \label{S:numerical} 

The computation of an approximation $\hat u_0$ to $u_0$, required in {\bf Step 1} of the
algorithm, can be carried out by standard finite element Galerkin schemes.
Depending on our knowledge on $f$ one can resort to known a priori estimates 
for $\e_1$, or may employ standard a posteriori estimates to ensure
that the underlying discretization provides a desired target accuracy.
Therefore, in the remainder of this section, we focus on a numerical implementation of {\bf Step 2} 
of the proposed algorithm.

Our proposed numerical algorithm for {\bf Step 2} is to use finite element methods to generate the approximations
$\hat \phi_j$ of the Riesz representers $\phi_j$.  Note that each of 
the functions $\phi_j$ is harmonic on $\Omega$  but we do not
require that the sought after numerical approximation $\hat \phi_j$ is itself harmonic but only that it provides an accurate  $H^1(\Omega)$ approximation to $\phi_j$.  This allows us to use finite element approximations which are themselves not harmonic.
However, the $\hat \phi_j$ will necessarily have to  be close to being harmonic since they approximate a harmonic function in the $H^1(\Omega)$ norm.  

 Our numerical approach to constructing a $\hat \phi_j$, discussed in \S\ref{SS:gal}, 
is to use discretely harmonic finite elements.   Here, $\hat \phi_j$ is   a discrete harmonic
extension   of a finite element approximation to the trace $\psi_j=T(\phi_j)$ computed by
solving a Galerkin problem.  In order to reduce computational cost (see Remark \ref{rem:cost}), we incorporate discrete harmonicity
as constraints and  introduce in \S\ref{SS:saddle} an {\em equivalent
saddle point formulation} that has the same solution $\hat \phi_j$, and which is the one
that we practically employ in the numerical experiments given in \S\ref{SS:Rr}. 
We  give in \S\ref{SS:apriori} an a priori analysis
with error bounds for $\| \phi_j-\hat \phi_j\|_{H^1}$ in terms of the 
finite element mesh size, in the case where the measurement functionals 
are continuous on $H^1(\Omega)$. These error bounds can in turn be used to ensure the prescribed
accuracy $\e_2$ in {\bf Step 2}. We finally discuss in \S\ref{SS:pointvalues} 
the extensions to the point value case where pointwise error bounds on $|\hat \phi_j(x_i)-\hat \phi_j(x_i)|$
are also needed.

In order to simplify notation, we describe these procedures for finding 
an approximation $\hat \phi$ to the Riesz representer $\phi\in \cH^s=\cH^s(\Omega)$ 
of a given linear functional $\nu$ on $\cH^s$. This numerical procedure is then applied with $\nu=\lambda_j$, to find the numerical approximations $\hat \phi_j$ to the Riesz representer $\phi_j$. 

For simplicity, throughout this section, we work under the assumption that
$\Omega$ is a polygonal domain  of $\R^2$ or polyhedral domain of $\R^3$.  This allows us to define finite element
spaces based on triangular or simplicial partitions of $\Omega$ that in turn induce similar partitions
on the boundary. We assume that $\frac 1 2 < s<\frac 3 2$, which is the relevant
range for such domains, as explained in \S \ref{S:hs}. Our analysis can be extended to more general 
domains with smooth or piecewise smooth boundaries, for example by using isoparametric elements 
near the boundary, however at the price of considerably higher technicalities.

\subsection{A Galerkin formulation}
\label{SS:gal}

Let $s>1/2$ be fixed and assume that $\nu$ is any linear form continuous on $\cH^s(\Omega)$ with norm 
\be  
\label{Cs}
C_s:=\max \{\nu(v) \; : \; \|v\|_{\cH^s(\Omega)}=1 \}
\ee
In view of the definition of the
$\cH^s$ norm, the representer $\phi\in \cH^s(\Omega)$ of $\nu$ for the
corresponding inner product
can be defined as
$$
\phi=E\psi,
$$
where $E$ is the harmonic extension operator of \eref{defE} below and where $\psi\in H^s(\Gamma)$ is the solution to the following variational problem:
\begin{equation}\label{exactE}
\<\psi,\eta\>_{H^s(\Gamma)}=\mu(\eta):=\nu(E\eta), \quad \eta\in H^s(\Gamma).
\end{equation}
 Note that this problem admits a unique solution and we have
$$
\|\psi\|_{H^s(\Gamma)}=\|\phi\|_{\cH^s(\Omega)}=C_s.
$$
Recall that
\be 
\label{defE}
Eg:={\rm argmin}\{\|\nabla v\|_{L^2(\Omega)} \; : \; v_\Gamma=g\}.
\ee  
The function $Eg$ is characterized by $T(Eg)=g$ and
$$
\int_\Omega \nabla Eg \cdot \nabla v=0, \quad v\in H^1_0(\Omega).
$$
From the left inequality in \eref{LMbound}, one has
\be
\|Eg\|_{H^1(\Omega)}\leq C_E\|g\|_{H^{1/2}(\Gamma)}, \quad g\in H^{1/2}(\Gamma),
\label{stabE}
\ee
where $C_E$ can be taken to be the inverse of the constant $c_0$ in \iref{LMbound}. 


Therefore, one approach to discretizing this problem is the following: consider 
finite element spaces ${\mathbb V}_h$ associated to a family of 
meshes $\{\cT_h\}_{h>0}$ of $\Omega$,
where as usual $h$ denotes the maximum meshsize.
We define ${\mathbb T}_h$ to be the space obtained by restriction
of ${\mathbb V}_h$ on the boundary $\Gamma$, that is,
$$
{\mathbb T}_h=T({\mathbb V}_h)
$$
Since we have assumed that $\Omega$ is
a polygonal or polyhedral domain, the space ${\mathbb T}_h$ is a standard finite element space for the boundary mesh. 
Having also assumed that $s<3/2$, when using standard $H^1$ conforming finite elements such as ${\mathbb P}_k$-Lagrange finite elements, 
we are ensured that ${\mathbb T}_h\subset H^s(\Gamma)$. 
We  denote by
$$
{\mathbb W}_h:=\{v_h\in {\mathbb V}_h\; : \; T(v_h)=0\},
$$
the finite element space with homogeneous boundary conditions.

We define the discrete harmonic extension operator $E_h$ associated to ${\mathbb V}_h$ as follows : for $g_h\in {\mathbb T}_h$, 
$$
E_h g_h:={\rm argmin}\{\|\nabla v_h\|_{L^2(\Omega)} \; : \; v_h\in {\mathbb V}_h, \; T(v_h)=g_h\}.
$$
Note that $E_hg_h$ is not harmonic.
Similar to $E$, the function $E_h g_h$ is characterized by $T(E_hg_h)=g_h$ and 
$$
\int_\Omega \nabla E_hg_h \cdot \nabla v_h=0, \quad v_h\in {\mathbb W}_h.
$$
Then, we define the approximation $\phi_h\in {\mathbb V}_h$ to $\phi$ as
$$
\phi_h=E_h\psi_h,
$$
where $\psi_h\in {\mathbb T}_h$ is the solution to the following variational problem:
\be
\<\psi_h,g_h\>_{H^s(\Gamma)}=\mu_h(g_h):=\nu(E_hg_h), \quad g_h\in {\mathbb T}_h.
\label{galer}
\ee
Here we are assuming that, in addition to be defined on $\cH^s(\Omega)$, the functional $\nu$ is
also well defined on the space $\mathbb{V}_h$. We shall further consider separately two instances where
this is the case : (i) $\nu$ is a continuous functional on $H^1(\Omega)$ and (ii) $\nu$ is a point evaluation functional.

Note
that \eref{galer} is {\em not} the straightforward Galerkin
approximation of \eref{exactE}, since $\mu_h$ differs from $\mu$. This
complicates somewhat the further conducted convergence analysis.  The  numerical method we employ for computing
$\phi_h$ is to numerically solve an equivalent saddle point problem described below.

We apply the  strategy  \eref{galer} to $\nu:=\lambda_j$ for each $j$ and thereby obtain the
corresponding approximations $\hat \phi_j:=\phi_h\in \mathbb V_h$. Since {\bf Step 2} requires that we
guarantee the error $\|\phi_j-\hat \phi_j\|_{H^1}\leq \e_2$, our main goal in this section is to
establish a quantitative convergence bound for $\|\phi-\phi_h\|_{H^1}$. We also need
to establish a pointwise convergence bound for $|\phi(x)-\phi_h(x)|$ 
when considering the modified version of {\bf Step 2} in the case that the measurements are  point values.

Similar to $E$, it will be important in our analysis to control the stability of $E_h$
in the sense of a bound
\be
\|E_hg_h\|_{H^1(\Omega)}\leq D_E\|g_h\|_{H^{1/2}(\Gamma)}, \quad g_h\in {\mathbb T}_h,
\label{stabEh}
\ee
with a constant $D_E$ that is independent of $h$. However, such a uniform bound is not readily
inherited from the stability of $E$. As observed in \cite{DS}, 
its validity is known to depend on the existence of uniformly
$H^1$-stable  linear projections onto ${\mathbb V}_h$ preserving the homogeneous boundary
condition, that is, projectors $P_h$ onto ${\mathbb V}_h$ that satisfy
\be
P_h(H^1_0(\Omega))={\mathbb W}_h \quad{\rm and}\quad \|P_hv\|_{H^1(\Omega)}\leq B\|v\|_{H^1(\Omega)},\quad v\in H^1(\Omega),
\label{szproj}
\ee
for some $B$ independent of $h$. One straightforward consequence of this is that if $v\in H^1(\Omega)$ with $v|_{\Gamma} \in \mathbb T_h$ then $P_h(v)|_{\Gamma} = v|_{\Gamma}$.

We next show   that the existence of such projectors is sufficient to guarantee the stability of $E_h$. For this, suppose    \eref{szproj} holds and $g_h\in {\mathbb T}_h$.  Then $P_hEg_h\in {\mathbb V}_h$     and the trace of $P_hEg_h$ is equal to $g_h$. 
It follows that
$$
\begin{array}{lll}
\| E_hg_h-P_hEg_h\|_{H^1(\Omega)}&\leq& C_P
\| \nabla E_hg_h-\nabla P_hEg_h\|_{L^2(\Omega)}\\
&\leq& C_P\| \nabla E_hg_h\|_{L_2(\Omega)} +C_P\| \nabla P_hEg_h\|_{L^2(\Omega)},\\
&\leq& 2C_P\| P_hEg_h\|_{H^1(\Omega)},
\end{array}
$$
where  $C_P$ is the Poincar\'e constant for $\Omega$.  Here,      the last inequality follows from the minimizing property of $E_hg_h$.
Thus, by triangle inequality, one has
$$
\|E_hg_h\|_{H^1(\Omega)}\leq (1+2C_P)\| P_hEg_h\|_{H^1(\Omega)}\leq (1+2C_P)B\| Eg_h\|_{H^1(\Omega)}\leq (1+2C_P)BC_E\|g_h\|_{H^{1/2}(\Gamma)},
$$
which is \iref{stabEh} with $D_E=(1+2C_P)BC_E$.

The requirement of  uniformly stable projectors $P_h$ with the property \eref{szproj} is satisfied by projectors of Scott-Zhang type \cite{SZ} when the family
of meshes $\{\cT_h\}_{h>0}$ is shape regular, that is, when  all elements $T$ have
a uniformly bounded ratio between their diameters $h(T)$ and the diameter $\rho(T)$ of their inner circle. 
In other words, the shape parameter
\be
\label{shape}
\sigma=\sigma(\{\cT_h\}_{h>0}):=\sup_{h>0}\max_{T\in \cT_h}\frac {h(T)}{\rho(T)},
\ee
is finite.
In all that follows in the present paper, we work under such an assumption on the meshes $\cT_h$.  Therefore, \eref{stabEh} holds when  ${\mathbb V}_h$ is subordinate to such partitions.}

\subsection{A saddle point formulation}
\label{SS:saddle}

Before attacking the convergence analysis, we need to stress an important 
computational variant of the above described Galerkin method, that leads to
the same solution $\phi_h$. It is based on imposing
harmonicity via a Lagrange multiplier. For this purpose, we introduce the Hilbert space $X^s(\Omega)$
that consists of all $v\in H^1(\Omega)$ such that $v_\Gamma\in H^s(\Gamma)$, and equip it with the norm
$$
\|v\|_{X^s(\Omega)}:=\(\|v_\Gamma\|_{H^s(\Gamma)}^2+\|\nabla v\|_{L^2(\Omega)}^2\)^{1/2}.
$$
Then, the Riesz representer $\phi$ is equivalently determined as the solution 
of the saddle point problem: find $(\phi,\pi)\in X^s(\Omega)\times H^1_0(\Omega)$ 
such that
\be
\label{sp}
\begin{array}{lcl}
 a(\phi,v)+b(v,\pi)&=&\nu(v),\quad    v\in X^s(\Omega) \\
b(\phi,z) &=& 0, \qquad \ z\in H^1_0(\Omega),
\end{array}
\ee
where the bilinear forms are given by
$$
a(\phi,v):=\<\phi_\Gamma,v_\Gamma\>_{H^s(\Gamma)}\quad{\rm and}\quad
b(v,\pi):=\<\nabla v,\nabla \pi\>_{L^2(\Omega)}.
$$
Clearly the second equation in \eref{sp} means that $\phi$ is harmonic and testing the first equation with a $v\in \cH^s(\Omega)$
shows that $\phi$ is the Riesz representer of $\mu$.

This saddle point formulation is well-posed: the bilinear forms $a$ and $b$ obviously satisfies the continuity properties
$$
a(\phi,v)\leq \|\phi_\Gamma\|_{H^s(\Gamma)}\|v_\Gamma\|_{H^s(\Gamma)}\leq \|\phi\|_{X^s(\Omega)}\|v\|_{X^s(\Omega)}, \quad \phi,v\in X^s(\Omega),
$$
and for the standard norm $\|v\|_{H^1_0(\Omega)}=\|\nabla v\|_{L^2(\Omega)}$,
$$
b(v,\pi)\leq \|\nabla v\|_{L^2(\Omega)} \|\nabla \pi\|_{L^2(\Omega)} \leq  \|v\|_{X^s(\Omega)} \|\pi\|_{H^1_0(\Omega)}, \quad v\in X^s(\Omega), \; \pi\in H^1_0(\Omega).
$$
In addition, for all $v\in \cH^s(\Omega)$, one has
$$
\|v\|_{X^s(\Omega)}^2\leq \|v_\Gamma\|_{H^s(\Gamma)}^2+\|v\|_{H^1(\Omega)}^2
\leq \|v_\Gamma\|_{H^s(\Gamma)}^2+C_E^2\|v\|_{H^{1/2}(\Gamma)}^2 
\leq (1+C_E^2)a(v,v),
$$
which shows that $a$ is coercive on the null space of $b$ in $X^s(\Omega)$.  Finally, the bilinear form $b$
satisfies the inf-sup condition
$$
\inf_{\pi\in {H^1_0(\Omega)}}\sup_{v\in {X^s(\Omega)}} \frac{b(v,\pi)}{\|v\|_{X^s(\Omega)}\|\pi\|_{H^1_0(\Omega)}}
\geq \inf_{\pi\in {H^1_0(\Omega)}}\frac{b(\pi,\pi)}{\|\pi\|_{X^s(\Omega)}\|\pi\|_{H^1_0(\Omega)}}=1.
$$
Therefore the standard LBB theory ensures existence and uniqueness of the solution pair $(\phi,\pi)$.

We now discretize the saddle point problem by searching for 
$(\phi_h,\pi_h)\in {\mathbb V}_h\times {\mathbb W}_h$ such that
$$
\begin{array} {llll}
& a(\phi_h,v_h)+b(v_h,\pi_h)&=\nu(v_h), \quad \quad &v_h\in {\mathbb V}_h\\
&b(\phi_h,z_h) &=0,  & z_h\in{\mathbb W}_h.
\end{array}
$$

\begin{remark}
The equivalence with the previous derivation of $\phi_h$ by the Galerkin approach is easily checked: the second equation
tells us that the solution $\phi_h$ is discretely harmonic, and therefore equal to 
$E_h\psi_h$ for some $\psi_h\in {\mathbb T}_h$. Then taking $v_h$ of the form $E_hg_h$ for $g_h\in {\mathbb T}_h$
gives us exactly the Galerkin formulation \iref{galer}.
\end{remark}

This discrete saddle point problem is uniformly well-posed when we equip the space ${\mathbb W}_h$
with the $H^1_0$ norm, and the space ${\mathbb V}_h$
with the $X^s$ norm. The continuity of $a$ and $b$, and the inf-sup 
condition for $b$ follow by the exact same arguments applied to the finite
element spaces, with the same constants. On the other hand, we need to 
check the uniform ellipticity of $a$ in the space ${\mathbb V}_{h}^\cH\subset {\mathbb V}_h$
of discretely harmonic functions, which can be defined as
$$
{\mathbb V}_{h}^\cH:=\{v_h\in {\mathbb V}_{h} \; :\; b(v_h,z_h)=0, \; z_h\in {\mathbb W}_{h}\},
$$
or equivalently as the image of ${\mathbb T}_h$ 
by the operator $E_h$. For all $v_h\in {\mathbb V}_{h,H}$
and $g_h=T(v_h)$, we write
$$
\|v_h\|_{X^s(\Omega)}^2\leq \|g_h\|_{H^s(\Gamma)}^2+\|v_h\|^2_{H^1(\Omega)}
\leq \|g_h\|_{H^s(\Gamma)}^2+D_E^2\|g_h\|_{H^{1/2}(\Gamma)}^2 
\leq (1+D_E^2)a(v_h,v_h),
$$
where we have used the discrete stability of $E_h$. 

\begin{remark}
\label{rem:cost}
In practice, we use this discrete saddle point formulation
for the computation of $\phi_h$ rather than the equivalent Galerkin formulation \iref{galer}
for the following reason. Let $N_h:= {\rm dim}\,\mathbb{V}_h$, $M_h
:= {\rm dim}\,\mathbb{W}_h$, and $P_h:= {\rm dim}\,\Tr = N_h-M_h$.
Computing the right hand side load vector in \eref{exactE} requires computing
discretely harmonic extensions of $P_h$ basis functions,
which means solving $P_h$ linear systems of dimension $M_h$.
In addition one has to solve the sparse linear system \eref{galer} of size 
$P_h$ followed by another system of size $M_h$ to compute $\phi_h=E_h\psi_h$. Using optimal iterative solvers of linear complexity
the minimum amount of work needed to compute one representer scales then like
$$
P_hM_h \sim   N_h^{1+ \frac{d-1}{d}} .
$$
while solving the saddle point problem requires the order of $N_h + M_h\sim N_h$ operations. 
On the other hand the characterization of $\phi_h$ through
\iref{galer} appears to be more convenient when deriving error bounds
for $\|\phi-\phi_h\|_{H^1(\Omega)}$. This is the objective of the next sections.
\end{remark}

\subsection{Preparatory results} 
\label{SS:ellipticGamma}

In the derivation of error bounds for $\|\phi-\phi_h\|_{H^1(\Omega)}$, we will
need several ingredients.

The first is the following lemma that quantifies the
perturbation induced by using $E_h$ in place of $E$. 

\begin{lemma}
\label{lemmaEEh}
For any $g_h\in {\mathbb T}_h$, one has
\be
\|(E-E_h)g_h\|_{H^1(\Omega)} \leq C_2h^{r-1} \|g_h\|_{H^s(\Gamma)}.
\label{EEh}
\ee
where  $C_2$ 
depends on $r$ and $s$, the shape-parameter $\sigma$, and on the geometry of $\Omega$.
\end{lemma}

\noindent
\begin{proof}
From the properties of $E$ and $E_h$, one has
$$
\<\nabla (E-E_h)g_h,\nabla v_h\>=0, \quad v_h\in {\mathbb W}_h
$$
This orthogonality property shows that 
$$
\|\nabla (Eg_h-E_hg_h)\|_{L^2(\Omega)}  \leq
\|\nabla (Eg_h-E_hg_h-v_h)\|_{L^2(\Omega)}, \quad v_h\in {\mathbb W}_h,
$$
and therefore 
$$
\|\nabla (Eg_h-E_hg_h)\|_{L^2(\Omega)}  \leq \min_{v_h\in {\mathbb V}_h,T(v_h)=g_h} \|\nabla(Eg_h-v_h)\|_{L^2(\Omega)}
\leq \|\nabla(Eg_h-P_hEg_h)\|_{L^2(\Omega)},
$$
where $P_h$ is the stable projector that preserves homogeneous boundary condition, see \iref{szproj}. 
It follows that
$$
\|\nabla (Eg_h-E_hg_h)\|_{L^2(\Omega)}
\leq (1+B)\min_{v_h\in {\mathbb V}_h} \|Eg_h-v_h\|_{H^1(\Omega)},
$$ 
where $B$ is the uniform $H^1$-stability bound on $P_h$.
By standard finite element approximation estimates and \iref{hshr}, we have
$$
\min_{v_h\in {\mathbb V}_h} \|Eg_h-v_h\|_{H^1(\Omega)}  \leq Ch^{r-1} \|Eg_h\|_{H^r(\Omega)} \leq CC_1h^{r-1}\|g_h\|_{H^s(\Gamma)},
$$
where the constant $C$ depends on $r$ and on the shape parameter $\sigma$.
The estimate \iref{EEh} follows by Poincaré inequality since $Eg_h-E_hg_h\in H^1_0(\Omega)$. \end{proof}

The second ingredient concerns the regularity of the solution
to the variational problem
\be
\<\kappa,v\>_{H^s(\Gamma)}=\gamma(v),\quad v\in H^s(\Gamma).
\label{hsvargamma}
\ee
For a general linear functional $\gamma \in H^{-s}(\Gamma)$, that is,  continuous on $H^s(\Gamma)$, we are only ensured 
that the solution $\kappa$ is bounded in $H^s(\Gamma)$, with $\|\kappa\|_{H^s(\Gamma)}=\|\gamma\|_{H^{-s}(\Gamma)}$. 
However, if $\gamma$ has some extra regularity, this then translates into additional regularity of $\kappa$.
 
As a simple example, consider the case where $\gamma$ is in addition continuous on $L^2(\Gamma)$, that is 
\be
\gamma(v)=\<g,v\>_{L^2(\Gamma)},
\label{gammag}
\ee
for some $g\in L^2(\Gamma)$, and assume that we work with $s=1$ and a polygonal domain. Then the variational problem
has a solution $\kappa\in H^1(\Gamma)$ and in addition $\kappa\in H^2(E)$ for each edge $E$ with weak second derivative
given by
$$
-\kappa''=g-\kappa\in L^2(\Gamma).
$$
In turn, standard finite element approximation estimates yield
$$
\min_{\kappa_h\in {\mathbb T}_h} \|\kappa-\kappa_h\|_{H^1(\Gamma)} \leq Ch \|g\|_{L_2(\Gamma)},
$$
with a constant $C$ that depends on the shape parameter $\sigma$.

Of course, gain of regularity theorems for elliptic problems are known in various contexts.
However, we have not found a general treatment of gain of regularity that addresses the setting of this paper. In going forward, we do not wish to systematically explore this gain in regularity and approximability for more general values of $s$
and smoothness of $\gamma$ since this would  significantly enlarge the scope of this paper. Instead, we state it as the following general assumption.\\

\noindent {\bf Assumption R:} for $s>\frac 12$ and $\delta>0$, there exists $r(s,\delta)>0$ such that
if $\gamma\in H^{-s+\delta}(\Gamma)$ for some $\delta>0$, 
then the solution $\kappa$ to \iref{hsvargamma}
satisfies
\be
\min_{\kappa_h\in {\mathbb T}_h} \|\kappa-\kappa_h\|_{H^s(\Gamma)} \leq Ch^{r(s,\delta)} \|\gamma\|_{H^{-s+\delta}(\Gamma)},
\label{approxetas}
\ee
with a constant $C$ that depends on $s$, $\delta$,  and on the shape parameter $\sigma$.

The above example shows that  $r(1,1)=1$ for a polygonal domain.} We expect that this assumption always holds for the range $\frac 1 2 <s<\frac 3 2$
that is considered here.

\subsection{An a priori error estimate for $\|\phi-\phi_h\|_{H^1(\Omega)}$}\label{SS:apriori}

In this section, we work under the assumption that the linear 
functional $\nu$ is continuous on $H^1(\Omega)$ with norm
$$
C_\nu:=\max \{\nu(v) \; : \; \|v\|_{H^1(\Omega)}=1 \}.
$$
Let us first check that this   assumption implies a uniform a priori bound on $\|\psi_h\|_{H^s(\Gamma)}$.  Indeed, we may write
$$
\|\psi_h\|^2_{H^s(\Gamma)}=\<\psi_h,\psi_h\>_{H^s(\Gamma)}=\nu(E_h\psi_h) \leq C_\nu D_E \|\psi_h\|_{H^{1/2}(\Gamma)}
\leq C_\nu D_E \|\psi_h\|_{H^{s}(\Gamma)},
$$
 where the first inequality used \eref{stabEh}.   Therefore, 
\be
\|\psi_h\|_{H^s(\Gamma)} \leq C_\nu D_E.
\label{appsih}
\ee

We have seen in \S \ref{S:hs} that the function $\phi$ belongs to the standard Sobolev space $H^r(\Omega)$ for $r$ defined in \eref{definer}.  We use this $r$ throughout this section.
 From \eref{hshr}, there exists a constant $C_1$ such that
\be
\|Ew\|_{H^r(\Omega)} \leq C_1\|w\|_{H^s(\Gamma)},  \quad  w\in H^s(\Gamma).
\label{Ers}
\ee
As noted in \S \ref{S:hs}, the amount of smoothness $r$ depends both on $s$ and on the geometry of $\Omega$. 
What is important for us is that since $s>1/2$, we have shown in Section~\ref{S:hs}  that $r>1$. For example, for smooth domains it is $r=s+\frac 1 2$.
The   fact that $\phi\in H^r(\Omega)$ hints that the finite element approximation $\phi_h$ to $\phi$ should converge
with a certain rate.

This is indeed the case as given in the following result.

\begin{theorem}
\label{theorate}
Under {\em \bf Assumption R}, we have
\be
\label{convphibound}
\|\phi-\phi_h\|_{H^1(\Omega)}\leq CC_\nu h^t,
\ee
where $t= \min\{r-1,r(s,s+ \frac 12)+r(s,s-\frac 12)\}$. The constant $C$ depends on $s$ and on the geometry of $\Omega$, and on the family of meshes through the shape parameter $\sigma$.
\end{theorem}

\begin{proof} We use the decomposition
\be
\phi-\phi_h=E\psi-E_h\psi_h=E(\psi-\psi_h) + (E-E_h)\psi_h,
\label{dec}
\ee
The second term can be estimated with the help of Lemma \ref{lemmaEEh} applied to
$g_h=\psi_h$ which gives
$$
\|(E-E_h)\psi_h\|_{H^1(\Omega)} \leq C_2h^{r-1} \|\psi_h\|_{H^s(\Gamma)}\leq C_2D_EC_{\nu} h^{r-1},
$$
from the a priori estimate \iref{appsih} for  $\psi_h$. We thus have obtained a bound in $\cO(h^{r-1})$ for the $H^1$ norm of the second term in \iref{dec}. 

For the first term, we know that
$$
\|E(\psi-\psi_h)\|_{H^1(\Omega)}\leq C_E\|\psi-\psi_h\|_{H^{1/2}(\Gamma)},
$$ 
and so we are led to estimate $\psi-\psi_h$ in the $H^{1/2}(\Gamma)$ norm. For this purpose, we introduce the intermediate 
 solution $\o \psi_h\in {\mathbb T}_h$ to the problem 
$$
\<\o \psi_h,g_h\>_{H^s(\Gamma)}=\mu(g_h)=\nu(Eg_h), \quad g_h\in {\mathbb T}_h,
$$
and we use the decomposition 
\be
\psi-\psi_h=(\psi-\o \psi_h) +(\o\psi_h-\psi_h).
\label{dec2}
\ee
We estimate the second term in \eref{dec2} by noting that for any $g_h\in {\mathbb T}_h$,
$$
\<\o \psi_h-\psi_h,g_h\>_{H^s(\Gamma)}=\nu((E-E_h)g_h) \leq C_\nu \|(E-E_h)g_h\|_{H^1(\Omega)}
\leq C_{\nu}C_2h^{r-1} \|g_h\|_{H^s(\Gamma)},
$$
where we have again used Lemma \ref{lemmaEEh}.
Taking $g_h=\o \psi_h-\psi_h$ we obtain a bound  $\cO(h^{r-1})$ for its $H^s(\Gamma)$ norm, and in turn
for its $H^{1/2}(\Gamma)$ norm.

It remains to estimate $\|\psi-\o \psi_h\|_{H^{1/2}(\Gamma)}$. 
Note that $\o \psi_h$ is exactly the Galerkin 
approximation of $\psi$ since we use the same linear form $\mu$ in both problems. In fact, we have
$$
\<\psi-\o \psi_h,g_h\>_{H^{s}(\Gamma)}=0, \quad g_h\in {\mathbb T}_h,
$$
that is $\o \psi_h$ is the $H^{s}$-orthogonal projection of $\psi$ onto ${\mathbb T}_h$
and therefore
$$
\|\psi-\o \psi_h\|_{H^{s}(\Gamma)}=\min_{\kappa_h\in {\mathbb T}_h} \|\psi-\kappa_h\|_{H^s(\Gamma)}.
$$
Since the linear form $\mu$ satisfies
$$
|\mu(g)|=|\nu(Eg)|\leq C_\nu \|Eg\|_{H^1(\Omega)}\leq C_\nu C_E \|g\|_{H^{1/2}(\Gamma)},
$$
and thus belongs to $H^{-1/2}(\Gamma)$,
we may apply the estimate \iref{approxetas} to $\gamma=\nu$, $\kappa=\psi$, $\delta=s-\frac 1 2>0$, to
reach 
\be  
\label{firstbound}
\|\psi-\o \psi_h\|_{H^{1/2}(\Gamma)}\le \|\psi-\o \psi_h\|_{H^{s}(\Gamma)} \leq CC_\nu C_Eh^{r(s,s-\frac 1 2)}.
\ee  
This proves the theorem for the value
$t=\min\{r-1,r(s,s-\frac 1 2)\}>0$.

We finally improve the value of $t$ by using  
a standard Aubin-Nitsche duality argument as follows.
We now take $\kappa$ to be the solution of \iref{hsvargamma} with 
$$
\gamma(v) = \langle \psi-\o\psi_h, v \rangle_{H^{1/2}(\Gamma)},\quad v\in H^{1/2}(\Gamma),
$$ 
where $\langle .,.\rangle_{H^{1/2}(\Gamma)}$ stands for the $H^{1/2}$ scalar product associated with the norm $\| . \|_{H^{1/2}(\Gamma)}$.
We then write
$$
\|\psi-\o\psi_h\|_{H^{1/2}(\Gamma)}^2=\<\psi-\o\psi_h,\psi-\o\psi_h\>_{H^{1/2}(\Gamma)}
=\<\kappa,\psi-\o\psi_h\>_{H^s(\Gamma)}=\<\kappa-\kappa_h,\psi-\o\psi_h\>_{H^s(\Gamma)},
$$
where the last equality comes from Galerkin orthogonality. It follows that 
$$
\|\psi-\o\psi_h\|_{H^{1/2}(\Gamma)}^2\leq \|\kappa-\kappa_h\|_{H^s(\Gamma)}\|\psi-\o\psi_h\|_{H^s(\Gamma)}
\leq Ch^{r(s,s+\frac 1 2)}\|\psi-\o\psi_h\|_{H^{1/2}(\Gamma)} \|\psi-\o\psi_h\|_{H^s(\Gamma)},
$$
where we have again used \iref{approxetas} now with $\delta=s+\frac 1 2$.
Using the already established estimate  \eqref{firstbound}, it follows that
$$
\|\psi-\o\psi_h\|_{H^{1/2}(\Gamma)} \leq CC_EC_\nu h^{\tilde t},
$$
with $\tilde t := r(s,s+ \frac 12)+r(s,s-\frac 12)$. 
With all such estimates, the desired convergence bound follows with
$t:=\min\{r-1, \tilde t\}$.
\end{proof}

\begin{remark}\label{rem:rate}
In the case of a polygonal domain and $s=1$ which is further considered in our numerical experiments, we know that
$r=\frac 3 2$ and $r(1,1)=1$ so that $\tilde t\geq r(1,\frac 3 2)\geq  1$. In turn the convergence bound is established with $t=r-1=\frac 1 2$. 
\end{remark}


 \subsection {The case of point value evaluations}
 \label{SS:pointvalues}

We discuss now the case where
$$
\nu(v)= \delta_z(v)=v(z),
$$
for some $z\in \o\Omega$. In order to guarantee that point evaluation is a continuous functional on $\cH^s$, we assume that 
$$
s> \frac{d-1} 2,
$$
that is $s> \frac 12$ for $d=2$, and  $s>1$ for $d=3$. We want to find the Riesz representer of such a point evaluation functional on $\cH^s$. Note that our assumption on $s$ ensures
the continuous embeddings 
$$
H^s(\Gamma)\subset C(\Gamma),
$$
as well as 
$$
\cH^s(\Omega)\subset H^r(\Omega)\subset C(\o\Omega),
$$
since in view of \eqref{definer}
$$
r=\min\Big\{s+\frac 1 2,r^*\Big\} >\frac d 2,
$$
where in the inequality we recall that $r^* > \frac 3 2$ for polygonal domains. 

 The point evaluation functional $\nu$ is thus continuous on $\cH^s(\Omega)$
with norm $C_s$ bounded independently of the position of $z$.
Of course, the Galerkin scheme analyzed above for $\nu\in 
H^1(\Omega)^*$ continues to make sense since $\nu$ is well defined
on the space ${\mathbb V}_h$.  

As explained in \S\ref{SS:point}, the prescriptions in Step 2 of the
recovery algorithm need to be strengthened in the point evaluation setting.
Thus, we are interested in bounding the pointwise error
$|\phi(x)-\phi_h(x)|$ at the measurement points,
in addition to the $H^1$-error $\|\phi-\phi_h\|_{H^{1}(\Omega)}$. In what follows, 
we establish a modified version of Theorem \ref{theorate} 
in the point value setting that gives a convergence rate for $\|\phi-\phi_h\|_{H^{1}(\Omega)}$,
and in addition for $\|\phi-\phi_h\|_{L_\infty(\Omega)}$
ensuring the pointwise error control. We stress that the numerical method remains
unchanged, that is, $\phi_h$ is defined in the exact same way as previously.
The new ingredients that are needed in our investigation are two classical results
on the behavior of the finite element method with respect to the $L_\infty$ norm.

The first one is the so-called weak discrete maximum principle which states that there exists
a constant $C_{\max}$ such that, for all $h>0$,
\be
\|E_h g_h\|_{L_\infty(\Omega)}\le C_{\max} \|g_h\|_{L_\infty(\Gamma)}, \quad g_h\in {\mathbb T}_h.
\label{discmax}
\ee
This result was first established in \cite{CR} with constant $C_{\max}=1$ for piecewise linear Lagrange finite elements
under acuteness assumptions on the angles of the simplices. The above  version with $C_{\max}\geq 1$ is established
in \cite{Schatz} for Lagrange finite elements of any degree on $2d$ polygonal domains, 
under the more general assumption that the meshes $\{\cT_h\}_{h>0}$
are quasi-uniform (in addition to shape regularity, all elements of $\cT_h$ have diameters of order $h$).
A similar result is established in \cite{LL} on $3d$ convex polyhedrons.

The second ingredient we need is a stability property in the $L_\infty$ norm of the Galerkin projection
$R_h: H^1_0(\Omega)\to {\mathbb W}_h$ where $R_hv$,  $v\in H_0^1(\Omega)$,  is defined by
$$
\int_\Omega \nabla R_hv \cdot\nabla v_h=\int_\Omega \nabla v \cdot\nabla v_h, \quad v_h\in {\mathbb W}_h.
$$
Specifically, this result states that there exists a constant $C_{\rm gal}$ and exponent $a\geq 0$ such that,
for all $h>0$,
\be
\|R_hv\|_{L_\infty(\Omega)}\leq C_{\rm gal} (1+|\ln(h)|)^{a} \|v\|_{L_\infty(\Omega)}, \quad v\in L_\infty(\Omega)\cap H_0^1(\Omega),
\label{stabritz}
\ee
that is, the Ritz projection is stable and quasi-optimal, uniformly in $h$, up to a logarithmic factor. 
This result is established in \cite{Schatz} for Lagrange finite elements on $2d$ polygonal domains
and quasi-uniform partitions, with $a=1$ in the case of piecewise linear elements and $a=0$ for higher order elements.
A similar result is established in \cite{LL} with $a=0$ for convex polygons and polyhedrons.
In going further, we assume that the choice of finite element meshes ensures 
the validity of \iref{discmax} and \iref{stabritz}.

We begin our analysis with the observation that under the additional mesh assumptions, Lemma \ref{lemmaEEh} can be 
adapted to obtain an estimate on $\|(E-E_h)g_h\|_{L_\infty(\Omega)}$. 

\begin{lemma}
\label{lemmaEEhinf}
For any $g_h\in {\mathbb T}_h$, one has
\be
\|(E-E_h)g_h\|_{L_\infty(\Omega)} \leq C_3(1+|\ln(h)|)^{a})h^{r-\frac d 2} \|g_h\|_{H^s(\Gamma)},
\label{EEhinf}
\ee
where $C_3$ 
depends on $(r,s)$, the geometry of $\Omega$, and the family of meshes through $C_{\rm gal}$.
\end{lemma}

\noindent
\begin{proof} 
For
any $v_h\in {\mathbb V}_h$ such that $T(v_h)=g_h$, we write
$$
\|(E-E_h)g_h\|_{L_\infty(\Omega)}\leq \|Eg_h-v_h\|_{L_\infty(\Omega)}+\|E_hg_h-v_h\|_{L_\infty(\Omega)}.
$$
It is readily seen that $E_hg_h-v_h=R_h(E_hg_h-v_h)=R_h(Eg_h-v_h)$. 
Indeed $R_hE_hg_h- R_hEg_h  \in \mathbb{W}_h$ and $\int_\Omega \nabla (R_h(E_hg_h- Eg_h))\cdot \nabla v_h=
\int_\Omega \nabla  (E_hg_h- Eg_h)\cdot \nabla v_h =0$ for all $v_h\in \mathbb{W}_h$. Therefore, by \iref{stabritz}, we obtain
$$
\|(E-E_h)g_h\|_{L_\infty(\Omega)}
\leq 
(1+C_{\rm gal} (1+|\ln(h)|)^{a}) \min_{v_h\in {\mathbb V}_h,T(v_h)=g_h}\|Eg_h-v_h\|_{L_\infty(\Omega)}.
$$
On the other hand, we are ensured that $Eg_h$ belongs to $H^r(\Omega)$ where $r>\frac d2$,
and therefore has H\"older smoothness of order $r-\frac d 2>0$ with
$$
\|Eg_h\|_{C^{r-\frac d 2}(\Omega)}\leq C_e\|Eg_h\|_{H^r(\Omega)} \leq C_eC_1\|g_h\|_{H^s(\Gamma)},
$$
where $C_e$ is the relevant continuous embedding constant. By standard finite element approximation theory,
$$
\min_{v_h\in {\mathbb V}_h,T(v_h)=g_h}\|Eg_h-v_h\|_{L_\infty(\Omega)} \leq C h^{r-\frac d 2} \|Eg_h\|_{C^{r-\frac d 2}(\Omega)},
$$
where $C$ depends on $r$ and the shape-parameter $\sigma$ and therefore we obtain \iref{EEhinf}. 
\end{proof}

We are now in position to give an adaptation of Theorem \ref{theorate} to the point value setting.

\begin{theorem}
\label{theorateinf}
Under {\em \bf Assumption R}, for 
any $t_1<\min\{r-\frac d 2,r(s,s+ \frac 12)+r(s,s-\frac 12)\}$, one has
\be
\label{convphibound2}
\|\phi-\phi_h\|_{H^{1}(\Omega)} \leq Ch^{t_1},
\ee
and for any $t_2<\min\{r-\frac d 2,2r(s,s-\frac {d-1} 2)\}$, one has
\be
\label{convphiboundinf}
\|\phi-\phi_h\|_{L_\infty(\Omega)}\leq Ch^{t_2}.
\ee
The constant $C$ depends in both cases on $s$, $t_1$ and $t_2$, on the geometry of $\Omega$,
as well as on the family of meshes through the constants $C_{\max}$ and $C_{{\rm gal}}$, and the shape parameter $\sigma$.
\end{theorem}

\begin{proof}
We estimate $\|\phi-\phi_h\|_{H^{1}(\Omega)}$ by adapting certain steps in the proof of Theorem \ref{theorate}.
The first change lies in the a priori estimate of the $H^s(\Gamma)$ norm
of $\psi_h$ that was previously given by \iref{appsih} which is not valid anymore since $C_\nu=\infty$. Instead, we write
$$
\|\psi_h\|^2_{H^s(\Gamma)}=\<\psi_h,\psi_h\>_{H^s(\Gamma)}=\nu(E_h\psi_h) 
\leq \|E_h\psi_h\|_{L_\infty(\Omega)}
\leq C_{\max}\|\psi_h\|_{L_\infty(\Gamma)}
\leq C_{\max}B_s\|\psi_h\|_{H^{s}(\Gamma)},
$$
where we have used \iref{discmax} and where $B_s$ is the continuous embedding constant
between $H^{s}(\Gamma)$ and $L_\infty(\Gamma)$. In turn, we find that
\be
\|\psi_h\|_{H^s(\Gamma)} \leq C_{\max}B_s,
\label{appsih2}
\ee
which results in the slightly modified estimate
$$
\|(E-E_h)\psi_h\|_{H^1(\Omega)} \leq C_2C_{\max}B_s h^{r-1},
$$
for the second term of \iref{dec}.

For the  first term $E(\psi-\psi_h)$, we proceed in a similar manner to the proof of Theorem \ref{theorate}.  Namely, we 
estimate the $H^{1/2}(\Gamma)$ norms of two summands in \iref{dec2}. The estimate
of $\|\o \psi_h-\psi_h\|_{H^{1/2}(\Gamma)}$ is modified as follows. We note that for any $g_h\in {\mathbb T}_h$,
$$
\<\o \psi_h-\psi_h,g_h\>_{H^s(\Gamma)}=\nu((E-E_h)g_h) \leq
\|(E-E_h)g_h\|_{L_\infty(\Omega)}
\leq  C_3(1+|\ln(h)|)^{a})h^{r-\frac d 2} \|g_h\|_{H^s(\Gamma)},
$$
where we have now used Lemma \ref{lemmaEEhinf}.
Taking $g_h=\o \psi_h-\psi_h$ we obtain a bound of
order $\cO(h^{r-\frac d 2})$ up to logarithmic factors 
for its $H^s$ norm, and in turn for its $H^{1/2}$ norm.
The estimate of $\|\psi-\o \psi_h\|_{H^{1/2}(\Gamma)}$ is left unchanged and of order $\cO(h^{\tilde t})$.
Combining these various estimates, we have established \iref{convphibound2} for any $t_1<\min\{r-\frac d 2,\tilde t\}$,
with $\tilde t := r(s,s+ \frac 12)+r(s,s-\frac 12)$. 

We next estimate $\|\phi-\phi_h\|_{L_\infty(\Omega)}$ by
the following adaptation of the proof of Theorem \ref{theorate}. For the first term $(E-E_h)\psi_h$ of 
\iref{dec} we use Lemma \ref{lemmaEEhinf} combined with the estimate \iref{appsih2} of $\psi_h$
which give us
$$
\|(E-E_h)\psi_h\|_{L_\infty(\Omega)} \leq C_{\max}B_s C_3(1+|\ln(h)|)^{a})h^{r-\frac d 2}.
$$
For the second term $E(\psi-\psi_h)$, we use the continuous maximum principle to obtain
$$
\|E(\psi-\psi_h)\|_{L_\infty(\Omega)}
\leq \|\psi-\psi_h\|_{L_\infty(\Gamma)}
\leq \|\psi_h-\o \psi_h\|_{L_\infty(\Gamma)}+\|\psi-\o \psi_h\|_{L_\infty(\Gamma)}
$$
For the first summand, we write
$$
\|\psi_h-\o \psi_h\|_{L_\infty(\Gamma)}\leq C_e\|\psi_h-\o \psi_h\|_{H^s(\Gamma)},
$$
where $C_e$ is the relevant continuous embedding constant, and we have already observed
that $\|\psi_h-\o \psi_h\|_{H^s(\Gamma)}$ satisfies a bound in $\cO(h^{r-\frac d 2})$ up to logarithmic factors.
For the second summand, we may write 
$$
\|\psi-\o \psi_h\|_{L_\infty(\Gamma)}\leq C_e\|\psi-\o \psi_h\|_{H^{s}(\Gamma)},
$$
where $C_e$ is the relevant continuous embedding constant. 
Since  $\nu$ belongs to $H^{-s+\delta}(\Gamma)$ for all $\delta<s-\frac {d-1} 2$,
we can apply the estimate \iref{approxetas} to
reach a convergence bound
$$
\|\psi-\o \psi_h\|_{H^{s}(\Gamma)} \leq Ch^{r(s,\delta)},
$$
where $C$ depends on the closeness of $\delta$ to $s-\frac {d-1} 2$, and on the family of meshes through the 
shape parameter $\sigma$. Combining these estimates then gives \iref{convphiboundinf} for any $t_2<\min\{r-\frac d 2,
\tilde t\}$ where $\tilde t=r(s,s-\frac{d-1}{2})$, since $\delta$ can be picked arbitrarily close to $s-\frac {d-1} 2$.

 We can improve the range of $t_2$ as follows: pick any $\o s$ such that $\frac {d-1} 2<\o s< s$
and write
$$
\|\psi-\o \psi_h\|_{L_\infty(\Gamma)}\leq C_e\|\psi-\o \psi_h\|_{H^{\o s}(\Gamma)},
$$
where $C_e$ is the relevant continuous embedding constant. We then apply a similar 
Aubin-Nitsche argument to derive an estimate 
$$
\|\psi-\o\psi_h\|_{H^{\o s}(\Gamma)} \leq C h^{r(s,\delta)+r(s,s-\o s)}.
$$
Combining these estimates gives \iref{convphiboundinf} for any $t_2<\min\{r-\frac d 2,\o t\}$,
where $\o t:=2r(s,s-\frac{d-1}{2})$ since $\o s$ can be picked arbitrarily close to $\frac {d-1} 2$ and $\delta$ arbitrarily close to $s-\frac {d-1} 2$.

\end{proof}

 \section{Numerical Illustrations}
 \label{SS:Rr}
  
  In this section, we implement some examples of our numerical method.  For this, we have
  to specify the domain $\Omega$, the functionals $\lambda_j$, and a function
  $u\in H^1(\Omega)$ which    gives rise to the data vector $w=\lambda(u)$.  While our numerical
  method can be applied to general choices for these quantities, in our illustrations we  make these choices
    so that   the computations are not too involved but yet allow us the flexibility to illustrate certain features of our algorithm.
  The specific choices we make for our numerical example are the following.\\
  
  \noindent{\bf The domain:}  
  In order to simplify the presentation, we restrict ourselves when $\Omega = (0,1)^2$   
  but point out again that the algorithm can be extended to more general domains. 
  \vskip .1in
  \noindent{\bf The function $u$:} For the function $u$ we choose 
     the harmonic function $u=u_\cH$ where
\be 
\label{defv}
u_\cH(x,y) = e^x \cos(y), \qquad (x,y) \in \Omega:= (0,1)^2.
\ee 
This choice means that $u_0=0$ and therefore allows us not to deal with the computation of $\hat u_0$.  This choice 
corresponds to the right side $f=0$. Note that the  trace of $u_\cH$ on the boundary $\Gamma$  
is piecewise smooth and continuous. Therefore, we have $T(u_\cH) \in H^1(\Gamma)$.  
We take $s=1$ as our assumption on the value of $s$.  This means that we
shall seek Riesz representor for the functionals given below when viewed as acting on $\cH^1(\Omega)$.
  
  \subsection{The case of linear functionals   defined on $H^1(\Omega)$}
  \label{SS:lfH1}
   In this section, we consider numerical experiments for linear functionals defined on $H^1(\Omega)$. In our illustrative example, we relabel
   these functionals by double indices associated with a regular square grid. More precisely,
\be
\label{lij}
\lambda_{i,j}(v) := \frac{1}{\sqrt{2\pi r^2}} \int_{\Omega} v(z) e^{-\frac{1}2 \frac{| z-z_{i,j}|^2}{r^2}}dz, \ v\in H^1(\Omega),\quad i,j=1,...,\sqrt{m}.
\ee
Here,
we assume  that $m$ is a square integer and $r=0.1$ in our simulations.
The centers $z_{i,j} \in \Omega$ are uniformly distributed
\be 
\label{defzij}
z_{i,j} := \frac 1 {\sqrt{m}+1} (i,j), \qquad i,j=1,...,\sqrt{m}.
\ee

Recall that our numerical algorithm as described in Section~\ref{ss:numerical_algorithm} is based on finite element methods.
 Specifically, we use the finite element spaces 
$$
\mathbb V_h:= \left\lbrace v_h \in C^0(\overline \Omega) \ : \ v_h |_T \in \mathcal Q^1, \qquad T \in \mathcal T_h\right\rbrace, 
$$
where $\mathcal T_h$ are subdivisions of $\Omega$ made of squares of equal side length $h$ and $\mathcal Q^1$ 
denotes the space of polynomials of degree at most 1 in each direction. In order to study the effect of the mesh-size we specifically
consider
$$
h=h_n:=2^{-n}, \quad n=4,\dots,9,
$$
that is, bilinear elements on uniformly refined meshes with mesh-size $2^{-n}$.

  We display in Table~\ref{t:recovery_error} the results of our numerical recovery algorithm.  
  The entries in the table are the recovery errors 
  $$
  e(m,n):=\| u_\cH - \hat u_\cH \|_{H^1(\Omega)},
  $$
  where $\hat u_\cH\in \mathbb V_{h_n} $ is the recovery for the particular values of $m$ and $n$.

\begin{table}[ht!]
    \centering
    \begin{tabular}{c||c| c| c| c| c }
    \diagbox{n}{m}     &  4 & 9 & 16 & 25 & 36 \\ \hhline{|=||=|=|=|=|=|}  
    4   &   0.7  & 0.28  &  0.2 & 141.73  & 49.43 \\ \hline
    5    &  0.7  &0.28  & 0.18  & 16.0  & 16.31 \\ \hline
    6    &  0.7  & 0.28  & 0.18  & 0.2  & 1.79\\ \hline
    7    &  0.7  & 0.28  & 0.18  & 0.16   & 0.11  \\ \hline
    8    &  0.7  & 0.28  & 0.18  & 0.09   & 0.06  \\ \hline
    9    &  0.7  & 0.28  & 0.18  &0.09  & 0.06\\ \hline
    \end{tabular}
    \caption{Recovery error $e(m,n)$ for different amounts of Gaussian measurements $m$ and finite element refinements $n$. 
    }
    \label{t:recovery_error}
\end{table}

We have proven in this paper that our numerical recovery algorithm is near optimal with constant $C$ that can be made arbitrarily close to one by choosing $n$ sufficiently large.  This means that the error $e(m,n)$ satisfies  $e(m,n)\le CR(K^\cH_w)_{H^1(\Omega)}$ for $n$ sufficiently large.
Increasing the number $m$ of measurements is expected to decrease this Chebyshev radius.  While one is tempted to think
that the entries in each column of the table provides an upper bound for the Chebyshev radius of $K^\cH_w$ for these measurements,
this is not guaranteed since we are only measuring the error for one function from $K_w$, namely $u_\cH$, and not all possible functions from $K_w$.  However, the entries in any given column provide a lower bound for the Chebyshev radius of $K^\cH_w$ provided $n$ is sufficiently large.

 Increasing the number $m$ of measurements   requires a finer resolution, i.e., increasing $n$,  of the finite element discretization
until the perturbation $\e$ in Theorem \ref{T:H2} is sufficiently small. This is indeed confirmed by the results in Table \ref{t:recovery_error} where stagnating error bounds (in each fixed column) indicate the corresponding tip-over point. We notice in particular that for small values of $n$, the error becomes very large as
$m$ grows. This is explained by the fact that the Gramian matrix $G$ becomes severely ill-conditioned, and 
in turn the prescriptions on $\|G-\hat G\|_1$ cannot be fulfilled when using finite element approximation of the Riesz representers
on too  coarse meshes. An overall convergence of the recovery error to zero can, of course, only take
place when both $m$ and $n$ increase.

\subsection{The case of point value measurements} \label{ss:point_num}
In this section, we describe our numerical experiments in the case where the linear functionals $\lambda_{i,j}$ are point
evaluations at points from $\overline \Omega$.  Recall that while the $\lambda_{i,j}$ are not defined for general functions in
$H^1(\Omega)$ they are defined for functions in the model class $K^\cH:=U(\cH^s(\Omega))$ provided $s$ is sufficiently large
($s>1/2$ for $d=2$ and $s>1$ for $d=3$).  This means that the optimal recovery problem is well posed in such a case.
We have given in \S\ref{SS:point} sufficient conditions on a numerical algorithm to give near optimal recovery
and then we have gone on to show in \S\ref{SS:pointvalues} that our proposed numerical algorithm based on discrete harmonics
converges to a near optimal recovery with any constant $C>1$ provided that the finite element spaces are discretized fine enough.

In the numerical experiments of this section, we again take $\Omega=(0,1)^2$, $s=1$, and the data to be the point values of the harmonic function $u_\cH$ defined in \eref{defv}.  We choose the
evaluation points to be the
$z_{i,j}$ of \eref{defzij}.
We now provide in Table~\ref{t:recovery_error_p} the recovery error $e(m,n)$.
The observed behavior is similar to the case of Gaussian averages; see Table~\ref{t:recovery_error}. 

\begin{table}[ht!]
    \centering
    \begin{tabular}{c||c| c| c| c| c }
    \diagbox{n}{m}     &  4 & 9 & 16 & 25 & 36 \\ \hhline{|=||=|=|=|=|=|}  
    4   &   0.70  & 0.28  &  0.19 & 14.43  & 15.49 \\ \hline
    5    &  0.70 &  0.28  &  0.18 & 32.56 & 8.02\\ \hline
    6    &  0.70 & 0.28  & 0.18  & 1.51  &  2.27\\ \hline
    7    &  0.70 & 0.28 &  0.18  & 0.53  &  0.89 \\ \hline
    8    &  0.70  & 0.28  & 0.18 & 0.20   & 0.14\\ \hline
    9    &  0.70  & 0.28 & 0.18  & 0.14 &  0.11 \\ \hline
    \end{tabular}
    \caption{Recovery error $e(m,n)$ for different amounts of 
   point evaluation measurements $m$ and refinements $n$. 
     }
    \label{t:recovery_error_p}
\end{table}

\subsection{Additional comments on the approximation of Riesz representers}
\label{SS:addednumerical}
 
We provide a little more information on the computation of the Riesz representers that may be of interest to the reader.
We work in the same setting as  in the previous sections.
Let us begin with the rate of convergence of our numerical approximations to the Riesz representers.

  \begin{figure}[ht!]
\begin{center}
\input{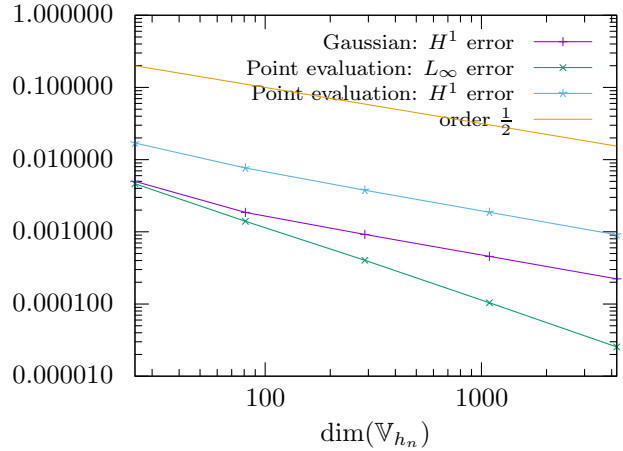}
\end{center}
\caption{Approximation errors for the Riesz representers of the Gaussian and point evaluation functionals.}
\label{f:conv}
\end{figure}

We first consider the computation of the Riesz representer for the Gaussian measurement functional centered at $z=z_{i,j}:=(0.75,0.5)$.  Let $\phi_n\in {\mathbb V}_{h_n}$ be the approximation to the Riesz representer $\phi$ produced by the finite element computation.  Figure~\ref{f:conv} shows the error $\|\phi_n-\phi_9\|_{H^1(\Omega)}$, $n=2,\dots,6$.  This graph indicates an error decay $Cn^{-1/2} = Ch_n$  (Theorem~\ref{theorate} only guarantees $Ch_n^{1/2}$, see also Remark~\ref{rem:rate}).

  Next consider the computation of the Riesz representer for point evaluation at the same $z$.  Figure~\ref{f:conv} reports the numerical computations of error in both the $H^1(\Omega)$ and $L_\infty(\Omega)$ norms. Again, the graph indicates an error decay $Ch_n$ for the $H^1(\Omega)$ norm and a decay rate closer to $Ch_n^2$ for the $L_\infty(\Omega)$ norm which are better than the rate guaranteed by Theorem~\ref{theorate} and Theorem~\ref{theorateinf}.

 \subsection{Convergence of the estimator}
\label{SS:addconverge}

We conclude the analysis of our algorithm with some remarks on the convergence of the estimator provided by our algorithm
as $m,n\to+\infty$.  We continue with  the case where $K=K^\cH=U(\cH^s(\Omega))$ with $s>1/2$ and $X$ is a Banach space for which $\cH^s(\Omega)$ embeds into $X$.
For each $m=1,2,\dots$, let 
$$\Lambda_m:=\{\lambda_{1,m},\dots,\lambda_{m,m}\}$$
be the set of measurement functionals. Recall that we are either  dealing with functionals in $H^1(\Omega)^*$ or with point evaluations which belong to $\cH^s(\Omega)^*$ by our assumption on $s$.
We assume that the data   vectors
\be 
\label{dataobservations}
w_m:= w_m(u):=(\lambda_{1,m}(u),\dots, \lambda_{m,m}(u)), \quad m=1,2,\dots,
\ee 
are observations of a fixed function $u\in K$.  
We are interested in what conditions on the sets $\Lambda_m$, $m=1,2,\dots$, ensure that the functions $\hat u_{m,n}$ produced by our algorithm converge in $\|\cdot\|_X$ to $u$ as $m,n\to\infty$
and whether this convergence is uniform over $u\in K$.

  Theorem \ref{T:H2} says that  
  \be 
  \label{be11}
  \|u-\hat u_{m,n}\|_X\le  R(K_{w_m(u)})_X+\e_{n,m},
  \ee 
  where for each fixed $m$, the error $\e_{n,m}\to 0$ as $n\to \infty$. 
  Therefore, if $R(K_{w_m(u)})_X\to 0$ as $m\to\infty$, then we know that given any error tolerance $\eta>0$,
  the error $\|u-\hat u_{m,n}\|_X$ will not exceed $\eta$ provided we take $m$ sufficiently large and then $n$ sufficiently large (depending on $m$).  Thus, we are left with the question of whether 
  $R(K_{w_m(u)})_X\to 0$, uniformly over $u\in K$, as the number $m$ of measurements   tends to $+\infty$.
In other words, we would like to know if 
\be 
\label{dense1}
R_m(K^\cH):=\max_{u\in K} R(K_{w_m(u)})_X\ \to 0,\quad m\to \infty.
\ee 
This is equivalent to  asking when does 
\be
\widetilde R_m(K^\cH):=\max \{\|v\|_{X} \; : \; v\in K^\cH \cap \cN_m\} \to 0,\quad{\rm as}\quad m\to \infty.
\label{dense}
\ee
where $\cN_m$ is the null space of $\lambda_{1,m},\dots,\lambda_{m,m}$.    Indeed, $ \widetilde R_m(K^\cH)\le 2R_m(K^\cH)\le 2\widetilde R_m(K^\cH)$. 
  Obviously, the validity of \eref{dense} requires some density assumption on the sets $\Lambda_m$ in $X^*$ as $m\to\infty$.   

We  illustrate how \eref{dense} is verified by considering   the two specific cases studied in our numerical experiments, i.e., the cases of Gaussian and pointwise
measurements on the Cartesian mesh points $z_{i,j}$ defined in \iref{defzij}.  
We   limit our discussion, as was done in our numerical examples, to the case when $X=H^1(\Omega)$, $\Omega=(0,1)^2$  and $K=K^\cH$, where $\cH=\cH^1(\Omega)$. It follows  that the functions in $K$ are in $U(H^{3/2}(\Omega))$ and hence (by the Sobolev embedding theorem) are not only continuous but in a ball of ${\rm Lip}~\alpha$ for each $\alpha<1/2$ with 
$$
\| f \|_{{\rm Lip}~\alpha}:= \max\left\{ \| f \|_{L_\infty(\Omega)}, \sup_{\substack{x,y\in \Omega \\ \ x \not = y}} \frac{| f(x)-f(y)|}{|x-y|^{\alpha}}\right\}.
$$
We recall that the measurement functionals in $\Lambda_m$  are associated with a grid  of points $z_{i,j}$ where $m=k^2$,    $k\geq 1$ an integer.  We denote by $\Omega_m=(h,1-h)\times (h,1-h)\subset \R^2$ the convex hull of these grid points.

Before going further, let us mention  the following remark, which we will use next.
\begin{remark}
\label{Rneeded}
Let   $R:=(a,b)\times(c,d)\subset \R^2$ be a rectangle in $\R^2$ and  $R_x:=R\setminus (b-h,b)\times[c,d]$. Let $g \in U(H^{1/2}(R))$ such that $\|g\|_{L_2(R_x)}\leq \varepsilon$. 
Note that  $\|g\|_{B^{1/2}_\infty(L_2(R))}\le \|g\|_{H^{1/2}(R)}\le 1$ due to the embedding of these spaces, and thus if we define $g_h:R_x\to\R$ by $g_h(x,y):=g(x+h,y)$, we have
$$
\|g_h-g\|_{L_2(R_x)} \leq   h^{1/2}.
$$ 
Then, the following bound on the norm of $g$ on the whole rectangle $R$ holds,
\begin{equation}\label{e:extension}
\|g\|_{L_2(R)}\leq \sqrt{2}(\varepsilon+h^{1/2}),
\end{equation}
since
\begin{eqnarray}
    \nonumber
\|g\|^2_{L_2(R)}&\leq&
\|g\|^2_{L_2(R_x)}+
\|g_h\|^2_{L_2(R_x)}\leq
\|g\|^2_{L_2(R_x)}+
\left [\|g_h-g\|_{L_2(R_x)}+\|g\|_{L_2(R_x)}\right ]^2\nonumber \\
&\leq& \varepsilon^2+\left [
 h^{1/2}+\varepsilon\right ]^2<
 2\left [
 h^{1/2}+\varepsilon\right ]^2.
 \nonumber
\end{eqnarray}
Similar statements hold if $R_x$ is replaced by $R\setminus(a+h,a)\times(c,d)$, $R\setminus(a,b)\times(d-h,d)$, or $R\setminus(a,b)\times(c,c+h)$.
\end{remark}

Now, we consider first the case of point evaluation.  In this case, any function $v$ appearing in the set of \eref{dense} is in $U(H^{3/2}(\Omega))$
and vanishes on a grid of points with spacing $h=m^{-1/2}$.  Standard Finite  Element estimation shows that
\be
\label{imp}
\|v\|_{H^1(\Omega_m)} \le C  m^{-1/4},
\ee
with $C$ an absolute constant. 
To extend the above estimate to $\Omega$, we use \eqref{e:extension} applied to appropriate rectangles obtained from $\Omega_m$ and with $h=m^{-1/2}$, $\varepsilon=Cm^{-1/4}$ (see \eref{imp}), and $g = v$, $g=\partial_xv$, or $g=\partial_yv$. Whence, we deduce that
$$
\|v\|_{H^1(\Omega)} \le C  m^{-1/4}
$$
for a different absolute constant $C$.



  We next consider the more intricate case of Gaussian measurements.
To prove \eref{dense},
we assume that \eref{dense} does not hold and derive a contradiction.  So, assume there is
a sequence of functions $v_m\in U(\cH^1(\Omega))\cap \cN_m$, $m=1,2, \dots $, that does not tend to $0$ in $H^1(\Omega)$.
This means that we can extract  an increasing sequence $(m_j)$ for which 
\be 
\label{subsequence}
\|v_{m_j}\|_{H^1(\Omega)} \ge \delta>0,\quad j=1,2,...,
\ee
for some $\delta>0$.
Since $\cH^1(\Omega)$ is continuously embedded into $H^{3/2}(\Omega)$, 
 and $H^{3/2}(\Omega)$ is compactly embedded in $H^s(\Omega)$ for each $1<s<\frac 3 2$, we can further extract a subsequence (which we continue to denote by $(m_j)$) such   that $v_{m_j}$
strongly converges towards a $v^*$ in $H^s(\Omega)$.  Note in particular that $v^*$ is a continuous function, in fact a Lip $s-1$ function on $\Omega$, and we have
\be 
\label{subsequence1}
\|v^*-v_{m_j}\|_{C(\Omega)}, \ \|v^*-  v_{m_j}\|_{H^1(\Omega)}\le A\| v^*-v_{m_j}\|_{H^s(\Omega)}  \to 0, \quad j\to\infty, 
\ee
with $A$ an absolute constant.  Hence, 
\be 
\label{norm1}  
\|v^*\|_{H^1(\Omega)}\ge \delta. 
\ee

Let us now examine how the measurement functionals act on any continuous function $v\in C(\Omega)$.  For any such $v $, we define
\be 
\label{functionals}
\tilde v:= g_r* \mathcal E v,  
\ee  
where 
$$
g_r(z):=\frac{1}{\sqrt{2\pi r^2}} e^{-\frac{1}2 \frac{| z|^2}{r^2}},
$$
is the Gaussian function used for the measurements and $ \mathcal E$ is the extension operator
by $0$ outside of $\Omega$.  The function $\tilde v$ is a continuous function on $\R^2$
and for any of our measurement functionals we have
\be 
\label{func2}
\lambda_{i,j}(v)=\tilde v(z_{i,j}), \quad 1\le i\le j.
\ee 
 
We claim that $\tilde v^*$ is identically zero on $\Omega$.  Indeed,  we have
\be 
\label{identzero}
|\tilde v^*(x)|\le |\tilde v^*(x)-\tilde v_{m_j}(x)|+ |\tilde v_{m_j}(x)-\tilde v_{m_j}(z)|+|\tilde v_{m_j}(z)|,\quad x,z\in\Omega.
\ee 
 Given any $x\in \Omega$ and $\e>0$, because of \eref{subsequence1}, we can make the first term in \eref{identzero} smaller than
 $\e/2$ for any sufficiently large $j$.  Then by again taking $j$ sufficiently large, there is a grid point $z$ for this $m_j$ that is  sufficiently close
 to $x$ so that the second term in \eref{identzero} is also smaller than $\e/2$.  Since $\tilde v_{m_j}(z)=0$ we see that
 $|\tilde v^*(x)|<\e$.  This proves that $\tilde v^*=0$.

Now that we know that $\tilde v^*$ is identically zero on $\Omega$, our final step is to show that this implies 
 that $v^*$ is identically zero on $\Omega$. Denoting by $\widehat f$ the Fourier transform of $f$, it follows that
$$
0=\int_\Omega \tilde v^* v^* =\int_{\R^2} \tilde v^*   \mathcal Ev^* = (2\pi)^{-2}\int_{\R^2}\hat g_r |\widehat{ \mathcal Ev^*}|^2.
$$
The positivity of $\hat g_r$ implies that $ \mathcal Ev^*=0$ and therefore $v^*=0$ on $\Omega$.  This contradicts \eref{norm1} and is our desired   contradiction.

Returning now to the general setting where $X$ and the measurement sets $\Lambda_m$ are general,  we will have that \eref{dense} holds only if the sets $\Lambda_m$ become  sufficiently dense in $X^*$ as $m$ gets large.  The precise rate of convergence of $R(K^\cH)$ towards $0$ will
depend on the prior class $\cH$ and on the choice of the 
 $\Lambda_m$, $m=1,2,\dots$. In the next section we discuss in more detail the choice of $\Lambda_m$, $m=1,2,\dots$,
in order to obtain the best achievable rate.

\section{Optimal data sites: Gelfand widths and sampling numbers}
\label{S:Gelfandwidths}

In this section, we make some comments on the number of measurements $m$ that are needed
to guarantee a prescribed error in the recovery of $u$.  Bounds on $m$ are known to be
governed by the Gelfand width for the case of general linear functionals and by sampling numbers
when the functionals are required to be point evaluations.  We explain what is known about these quantities for our specific model classes.  As we shall see these issues are not
completely settled for the model classes studied in this paper.  The problem of
finding the best choice of functionals, respectively point evaluations, is in need of further research.

We have seen that the accuracy of the optimal recovery of $u\in K_w$  is given by the Chebyshev radius $R(K_w):=R(K_w)_{H^1(\Omega)}$ or equivalently $R(K^\cH_w):=R(K^\cH_w)_{H^1(\Omega)}$ for the harmonic component.
The worst case recovery
error $R(K)$ over the class $K$ is defined by
\be 
R(K)_{H^1(\Omega)}:=\sup_{w\in\R^m}R(K_w)_{H^1(\Omega)},
\label{RK}
\ee 
Notice that this worst case error fixes the measurement functionals but allows the measurements $w$ to come from any function in $K$.
Both the individual error $R(K_w)$ and the worst case error $R(K)$ are very dependent on the choice of the data functionals $\lambda_j$.
For example, in the
case that these functionals  are point evaluations at points $z_1,\dots,z_m\in\bar \Omega$, then $R(K_w)$
and $R(K)$ will depend very much on the positioning of these points in $\bar\Omega$. 

In the case of general linear functionals, one may fix $m$ and then search for the $\lambda_1,\dots,\lambda_m$ that minimize  the worst case recovery error over
the class $K$. This minimal worst case error is called the {\it Gelfand width} of $K$. If we restrict the linear functionals to be given by    point evaluation, 
we could correspondingly search for the sampling points $x_1,\dots,x_m$ minimizing the worst case recovery error.  This  minimal 
error is called the {\it deterministic sampling number} of $K$. 

The goal of this section is not to
provide new results on Gelfand widths and sampling numbers, since we regard this as a separate issue in need of a systematic study, but to discuss what is known about them in our setting and refer the reader to the relevant papers.  
Let us recall that $R(K_w)$ is equivalent to $R(K^\cH_w)_{H^1}$ and so we restrict our discussion in what follows to sampling of harmonic
functions. 
  
\subsection {Optimal choice of functionals}
\label{SS:optimalfunctionals}  

Suppose we fix the number $m$ of observation to be allowed and ask  what is the optimal choice for the $\lambda_j$, $j=1,\dots,m$, and what is the optimal error of recovery for this choice.  The answer to the second question is given by the Gelfand width of $K$.  Given a compact set $K$ of  a Banach space $X$, we define the {\it Gelfand width} of $K$ in $X$ by    
\be 
\label{gelfand}
d^m(K)_{X}:= \inf_{\lambda_1,\dots,\lambda_m}  R(K)_X
\ee 
where the infimum is taken over the linear functionals defined on $X$. 
Let us mention that this definition 
differs from that employed in the classical literature \cite{P} where $d^m(K)_{X}$
is defined as the infimum over all spaces $F$ of codimension $n$
of $\max\{\|v\|_X \; : \; v\in K\cap F\}$. The two definitions are equivalent
in the case where $K$ is a centrally symmetric set such that $K+K\subset CK$ 
for some constant $C\geq 1$.

Any set of functionals which attains the infimum in \iref{gelfand} would be optimal.  
The Gelfand width is often used as a benchmark for performance since
it says that no matter how the $m$ functionals $\lambda_1,\dots,\lambda_m$ are chosen, 
the error of recovery of $u\in K$ cannot be better than $d^m(K)_X$.

When $X$ is a Hilbert space and $K$ is the ball of a Hilbert space $Y$ with compact embedding in $X$,
it is known that the Gelfand width coincides with the Kolmogorov width, that is
$$
d^m(K)_{X}=d_m(K)_X:=\inf_{\dim(E)=m} {\rm dist}(K,E)_X =\inf_{\dim(E)=m} \max\{\|v-P_E v\|_X \; : \; v\in K\},
$$
where the infimum is taken over all linear spaces $E$ of dimension $m$. This is precisely
our setting as  discussed in \S \ref{S: firstnumericalalgorithm}: taking $X=H^1:=H^1(\Omega)$ and $K$ as in \eref{classK}, 
we have
\be 
\label{equivalences}
d^m(K)_{H^1(\Omega)} = d^m(K^\cH)_{H^1(\Omega)}= d_m(K^\cH)_{H^1(\Omega)} \sim d_m(K^B)_{H^{1/2}(\Gamma)}=d^m(K^B)_{H^{1/2}(\Gamma)},  
\ee 
where the equivalence follows from \eref{LMbound}. Upper and  lower bounds 
for the Gelfand width of $K^B$ in $L_2(\Gamma)$ are explicitly given in \cite{Pez}. 

We can estimate the rate of decay of the Kolmogorov and Gelfand width of $K^B$ in $H^{1/2}(\Gamma)$ by 
the following general argument: as explained in \S\ref{SS:HsGamma}, for the admissible range of smoothness, the Sobolev spaces
$H^s(\Gamma)$ have an intrinsic description by locally mapping 
the boundary onto domains of $\R^{d-1}$. More precisely, in \cite{Nec} and \cite{Gr},
the $H^s(\Gamma)$ norm of $g$ is defined as
\be
\|g\|_{H^{s}(\Gamma)}:=\(\sum_{j=1}^J  \|g_j\|_{H^s(R_j)}^2\)^{1/2},
\label{sobnorm2}
\ee
where the $R_j$ are open bounded rectangles of $\R^{d-1}$
that are mapped by transforms $\gamma_j$ into portions $\Gamma_j$
that constitute a covering of $\Gamma$, and $g_j=g\circ \gamma_j$
are the local pullbacks. 

From this it readily follows that the Gelfand and Kolmogorov $m$-width
of the unit ball of $H^{s}(\Gamma)$ in the norm $H^{t}(\Gamma)$, with $0\leq t<s$
behaves similar to that of the unit ball of $H^{s}(R)$ in the norm $H^{t}(R)$ where $R$ is
a bounded rectangle of $\R^{d-1}$. The latter is known to behave
like $m^{-\frac{s-t}{d-1}}$. Therefore, for $K^\cH=U(\cH^s)$ with $s>\frac 1 2$ in the admissible range
allowed by the boundary smoothness, one has
\be
\label{gwkh}
c m^{-\frac{s-1/2}{d-1}}\le d^m(K^\cH)_{H^1(\Omega)}\le C m^{-\frac{s-1/2}{d-1}}, \quad m\ge 1,
\ee
where $c$ and $C$ are positive constants depending only on $\Omega$ and $s$.

\begin{remark} 
We have already observed in \S \ref{S:hs} that the space 
$\cH^s(\Omega)$ is continuously embedded in the Sobolev space
$H^r(\Omega)$ with $r:=\max\{s+\frac 1 2,r^*\}$ and in particular
$r=s+\frac 1 2$ for smooth domains. However the Gelfand 
and Kolmogorov widths of the unit ball of $H^r(\Omega)$ in $H^1(\Omega)$ 
have the slower decay rate $m^{-\frac{r-1}{d}}=m^{-\frac{s-1/2}{d}}$ 
compared to \iref{gwkh} for $\cH^s(\Omega)$. This improvement
reflects the fact that the functions from $\cH^s(\Omega)$ have $d$ variables
but are in fact determined by functions of $d-1$ variables.  The reduction in dimension from $d$ to $d-1$
is related to the fact that in our formulation of our problem we have complete knowledge of $f$.
\end{remark}

\subsection{Optimal choice of sampling points}
\label{SS:sampling}

We turn to the particular setting where 
the $\lambda_j$ are point evaluations functionals, 
$$
\lambda_j(v)=v(x_j),
$$
at $m$ points $x_j\in\o \Omega$.
Similar to the Gelfand width, the {\it deterministic sampling numbers} are defined as
\be 
\label{d:sampling}
\rho^m(K)_{X}:= \inf_{x_1,\dots,x_m}  R(K)_X,
\ee 
A variant of this is to measure the worst case expected
recovery error when the $m$ points are chosen at random according to a probability distribution
and search for the distribution that minimizes this error, leading to the 
{\it randomized sampling number} of $K$. Obviously, one has
\be
\rho^m(K)_{X}\geq d^m(K)_X.
\label{gelsa}
\ee

In the majority of the literature, deterministic and randomized sampling numbers are studied 
with error measured in the $L_2(\Omega)$ norm. In this setting, 
concrete strategies for optimal deterministic and randomized point design
have been given when $K$ is the unit ball 
of a reproducing kernel Hilbert space $H$ defined on $\Omega$.  In particular,
the recent results in \cite{MU,KU,NSU,DKU} show that under the assumption
$$
\sum_{m>0}|d^m(K)_{L^2(\Omega)}|^2<\infty,
$$
then, for all $t>\frac 1 2$,
$$
\sup_{m\geq 1}m^t d^m(K)_{L^2(\Omega)}<\infty \implies \sup_{m\geq 1}m^t \rho^m(K)_{L^2(\Omega)}<\infty.
$$
In words, under the above assumptions, optimal recovery in $L_2(\Omega)$ has the same algebraic convergence rate when 
using optimally chosen point values compared to an optimal choice of general linear functionals.

While similar general results have not been established for Gelfand width and sampling numbers
in the $H^1$ norm, we argue that they hold in our particular setting where $H=\cH^s(\Omega)$.
For simplicity, as in \S \ref{S:numerical}, we consider a domain that is either a polygon
when $d=2$ or polyhedron when $d=3$, and thus consider the range $\frac {d-1} 2<s<\frac 3 2$
where the restriction from  below ensures that $\cH^s(\Omega)\subset C(\o\Omega)$. Recalling
the finite element spaces ${\mathbb V}_h$ and their traces ${\mathbb T}_h$ on the boundary, 
based on quasi-uniform meshes $\{ \mathcal T_h\}_{h>0}$, we consider for a given $h>0$ the measurement points
$x_1,\dots,x_m$ that are the mesh vertices located on $\Gamma$. By the quasi-uniformity property
the number $m=m(h)$ of these points satisfies
$$
ch^{1-d} \leq m\leq C h^{1-d},
$$
for some $c,C>0$ independent of $h$. If $v\in \cH^s(\Omega)$, its trace $v_\Gamma$ belongs to
$H^s(\Gamma)$. Then, denoting by $I_h$ the piecewise linear interpolant on the boundary, 
standard finite element approximation theory ensures the estimate
$$
\|v_\Gamma-I_h v_\Gamma\|_{H^{1/2}(\Gamma)}\leq C h^{s-\frac 1 2}\|v_\Gamma\|_{H^s(\Gamma)}
=C h^{s-\frac 1 2}\|v\|_{\cH^s(\Omega)},
$$
for some $C$ that only depends on $s$. Therefore, introducing $\t v:=EI_h v$, one has
$$
\|v-\t v\|_{H^1(\Omega)}\leq C_E\|v_\Gamma-I_h v_\Gamma\|_{H^{1/2}(\Gamma)}
\leq CD_Em^{-\frac{s-1/2}{d-1}}\|v\|_{\cH^s(\Omega)}.
$$
Since $\t v$ only depends on the value of $v$ at the points $x_1,\dots,x_m$, we have thus proved
an upper bound of order $m^{-\frac{s-1/2}{d-1}}$ for $\rho^m(K^\cH)_{H^1(\Omega)}$, and
in turn for $\rho^m(K)_{H^1(\Omega)}$. In view of \iref{gelsa} and \iref{gwkh}, a lower bound
of the same order must hold. In summary, similar to the Gelfand widths, the sampling numbers satisfy
\be
\t c m^{-\frac{s-1/2}{d-1}}\le \rho^m(K)_{H^1(\Omega)}\le \t C m^{-\frac{s-1/2}{d-1}}, \quad m\ge 1,
\label{sampling}
\ee
where $\t c$ and $\t C$ are positive constants depending only on $\Omega$ and $s$.


\begin{thebibliography}{77}
 
 

\bibitem{Ad}
R. A. Adam and J. F. Fournier, {\it Sobolev spaces}, Elsevier, 2003.
 
\bibitem{AH}
B. Adcock and D. Huybrechs,
{\it Frames and numerical approximation}
SIAM Rev. 61(3), 443-473, 2019.

\bibitem{A} G. Auchmuty, 
{\it Reproducing Kernels for Hilbert Spaces of Real Harmonic Functions}, SINUM, {\bf 41}(5), 2009.




\bibitem{brunton2020machine}
S.L. Brunton, B.R. Noack, P. Koumoutsakos, {\it Machine Learning for Fluid Mechanics}, Annual Review of Fluid Mechanics {\bf 36}(9), 477--508, 2020.

\bibitem{CR}
P.G. Ciarlet and P.A. Raviart, 
{\it Maximum principle and uniform convergence for the finite element method},
Comput. Methods Appl. Mech. Engrg. {\bf 2}, 17—31, 1973.



\bibitem{devore2016data} R. DeVore, G. Petrova, and P. Wojtaszczyk, {\it Data assimilation in Banach spaces},
  arXiv preprint arXiv:1602.06342, 2016.

\bibitem{DKU} M. Dolbeault, D. Krieg and M. Ullrich, {\it A sharp upper bound for sampling numbers in 
$L^2$}, arXiv preprint arXiv:2204.12621, 2022.

\bibitem{DS} V. Domínguez and F.-J. Sayas, {\it Stability of discrete liftings}, 
C. R. Acad. Sci. Paris I {\bf 337}, 885-898, 2003.


\bibitem{duraisamy2019turbulence} 
K. Duraisamy, G. Iaccarino, and H. Xiao,  {\it Turbulence modeling in the age of data}, Annual Review of Fluid Mechanics {\bf 51}, 357--377, 2019.

\bibitem{formaggia2002numerical} 
L. Formaggia, J.-F. Gerbeau, F. Nobile, and A. Quarteroni,  {\it Numerical treatment of defective boundary conditions for the Navier--Stokes equations}, SIAM Journal on Numerical Analysis {\bf 40}(1), 376--401, 2002.


\bibitem{GI} A. N. Galybin  and  J. Ir\u sa, {\it On reconstruction of three-dimensional
harmonic functions from discrete data}, Proc. R. Soc. A {\bf 466}(2010), 1935–-1955.
   
\bibitem{Gr}
P. Grisvard, {\it Elliptic problems in non-smooth domains}, Pitman, 1985. 

\bibitem{grinberg2008outflow} 
L. Grinbergand G.E. Karniadakis,  {\it Outflow boundary conditions for arterial networks with multiple outlets}, Annals of biomedical engineering {\bf 51}, 1496--1514, 2008.

  

\bibitem{KU} 
D. Krieg and  M. Ullrich, {\it  Function Values Are Enough for $L_2$-Approximation }, JFoCM  {\bf 21}, 1141--1151, 2021.

\bibitem{LL}
D. Leykekhman and B. Li, {\it Weak discrete maximum principle of finite element methods in convex polyhedra},
Math. of Comp. {\bf 90}, 1-18, 2021.

\bibitem{MeM} A. Melkman and C. Micchelli, {\it Optimal Estimation of Linear Operators in Hilbert Spaces from Inaccurate Data}
SIAM Journal on Numerical Analysis {\bf 16}, 87-105, 1979.

\bibitem{MiR}   
C. Micchelli and T. Rivlin, {\it  A Survey of Optimal Recovery},  In: Micchelli C.A., Rivlin T.J. (eds) Optimal Estimation in Approximation Theory, 1--54, 1977.

\bibitem{MU}
M Moeller and T Ullrich, {\it $L_2$-norm sampling discretization and recovery of functions from RKHS with finite trace}
Sampling Theory, Signal Processing, and Data Analysis 19, 1-31, 2021.


\bibitem{Nec} J. Necas, {\it Direct methods in the Theory of Elliptic Equations} (english translation of the 1967 monograph), Springer 2012.


\bibitem{NSU} N Nagel, M Schäfer, T Ullrich, {\it A new upper bound for sampling numbers},
Foundations of Computational Mathematics 22, 445-468, 2022.

\bibitem{NW} 
E. Novak and H. Wozniakowski, {\it Tractability of Multivariate Problems}, Volume I: Linear Information. EMS Tracts in Mathematics, Vol. 6, Eur. Math. Soc. Publ. House, Z\"urich, 2008.

\bibitem {Pez}
I. Pesenson, Estimates of Kolmogorov, Gelfand, and linear widths on compact Riemannian manifolds,
Proceedings AMS, {\bf 144},   2985–2998, 2016.

 \bibitem{P} 
 A. Pinkus, {\it N-widths in Approximation Theory}, Vol. 7 of A Series of Modern Surveys in Mathematics, Springer Science \& Business Media, 2012.



\bibitem{raj}
K. Rajagopal, {\it On boundary conditions for fluids of the differential type}, in: A. Sequeira, Ed., Navier- Stokes Equations and Related Non-linear Problems, Plenum Press, New York, pp. 273-278, 1995.

\bibitem{richards2011appropriate} 
P.J. Richards, S.E. Norris, {\it Appropriate boundary conditions for computational wind engineering models revisited}, Journal of Wind Engineering and Industrial Aerodynamics {\bf 99} (4), 257--266, 2011.

\bibitem{richards2019appropriate} 
P.J. Richards, S.E. Norris, {\it Appropriate boundary conditions for computational wind engineering: Still an issue after 25 years}, Journal of Wind Engineering and Industrial Aerodynamics {\bf 190}, 245--255, 2019.
  

\bibitem{Schatz}
A.H. Schatz, 
{\it A Weak discrete maximum principle and stability of the finite element method in $L_\infty$
on plane polygonal domains. I}, Math. of Comp. {\bf 34}, 77-91, 1980.


\bibitem{SZ}
L. R. Scott and Sh. Zhang, 
{\it Finite element interpolation of nonsmooth functions satisfying boundary conditions},
Mathematics of computation, {\bf 54} (190), 483-493, 1990.



 
\bibitem{TW} 
J. Traub and H. Wozniakowski, {\it A General Theory of Optimal Algorithms}, Academic Press, 1980.
 
 
 
\bibitem{xu2021explore}
Hui Xu, Wei Zhang, Yong Wang, {\it Explore missing flow dynamics by physics-informed deep learning: The parameterized governing systems}, Physics of Fluids {\bf 33}, 095116, 2021.

\bibitem{yosida} 
K. Yosida, {\it Functional analysis}, Springer Science \& Business Media, 2012.
\vskip 1in



\noindent
Peter Binev, Department of Mathematics, University of South Carolina, Columbia,  SC 29208, binev@math.sc.edu
\vskip .1in

\noindent
Andrea Bonito, Department of Mathematics, Texas A\&M University, College Station, TX 77843, bonito@math.tamu.edu
\vskip .1in
\noindent
Albert Cohen, Laboratoire Jacques-Louis Lions,  Sorbonne Universi\'e, 4, Place Jussieu, 75005 Paris,  France, albert.cohen@sorbonne-universite.fr
\vskip .1in


\noindent
Wolfgang Dahmen, Department of Mathematics, University of South Carolina, Columbia,  SC 29208, dahmen@math.sc.edu
\vskip .1in

\noindent
Ronald DeVore, Department of Mathematics, Texas A\&M University, College Station, TX 77843, rdevore@math.tamu.edu
\vskip .1in

\noindent
Guergana Petrova, Department of Mathematics, Texas A\&M University, College Station, TX 77843, gpetrova@math.tamu.edu
\vskip .1in





\end{thebibliography}
\end{document}